%% file: ex_article_HAL.tex
\documentclass[hidelinks,onefignum,onetabnum]{siamart250211}

\input{ex_shared}

\ifpdf
\hypersetup{
  pdftitle={Averaged models for compressible two-phase stratified flows on thin domains},
  pdfauthor={P. Le Vourc'h, K. Saleh, and N. Seguin}
}
\fi

\begin{document}

\maketitle
\begin{abstract}
This paper deals with the derivation of compressible two-phase flow models. We use a thin domain approximation of a two-layer configuration governed by the Navier-Stokes equations, following the works [H. B. Stewart and B. Wendroff, J. Comp. Phys., 56 (1984)] and [V. H. Ransom and D. L. Hicks, {\em J. Comput. Phys.}, 75 (1988)]. In order to recover source terms and two-velocity models directly from this asymptotic analysis, we remove the continuity of the tangential component of the velocity at the interface, and properly scale friction and viscosity coefficients. We are then able to derive two-velocity one-pressure models, first in the barotropic case and then in the full Navier-Stokes-Fourier case.
\end{abstract}



\section{Introduction}
In this paper, we present a derivation of averaged models for compressible two-phase flows. The usual process consists in applying averaging operators on a complete local description of the flow, see for instance the classical textbooks~\cite{drew_theory_1999} and~\cite{ishii_thermo-fluid_2011}. The application of this process to bifluid formulations of stratified compressible two-phase flows led to a variety of different averaged models in the 1980s, see for instance~\cite{stewart_two-phase_1984,ransom_hyperbolic_1984,ransom_hyperbolic_1988}. The averaging process being incomplete, the previous papers provide additional heuristic arguments to close the models. They also discuss the influence of the closure arguments on the resulting models, in particular when focusing on interaction terms between phases. The \textit{ad hoc} closure conditions introduced in these papers --- especially at the interface --- make the derivation of two-velocity models possible. Such approachs have been followed in~\cite{demay_compressible_2017} to include gravity effects as well as mass and momentum transfers.

Recent applications of mathematical homogenization to multiphase compressible Navier-Stokes equations provided the first complete, and rigorous, derivation of one-dimensional averaged models including the mechanical interaction source terms~\cite{bresch_multi-fluid_2011,bresch_note_2015,bresch_compressible_2019,hillairet_baer-nunziato_2019,bresch_mathematical_2022,hillairet_analysis_2023}. Indeed, provided that the velocity field is smooth enough --- especially through the interfaces between fluids ---, this method enables the extraction of averaged variables and equations by letting the frequency of density oscillations tend to infinity. Since this procedure has yet only been carried in one space dimension, the inherent smoothness of the velocity field causes the averaged model to be described by a unique velocity field, making it difficult to obtain two-velocity models.

In this paper, we propose an asymptotic derivation of two-velocity, one-pressure averaged models. They are composed of two equations for phasic mass conservation, two nonconservative equations for phasic momentum, and two nonconservative equations for phasic energies:
\begin{equation}
    \begin{aligned}
        \partial_t (\alpha_k \rho_k)+\nabla \cdot (\alpha_k \rho_k \ubf_k) &= 0,\\
        \partial_t (\alpha_k \rho_k \ubf_k)+\nabla \cdot (\alpha_k \rho_k \ubf_k\otimes \ubf_k) +\nabla p_k &=S_k^m,\\
        \partial_t (\alpha_k \rho_k E_k)+\nabla \cdot (\alpha_k (\rho_k E_k+p_k) \ubf_k) &=S_k^e,  
    \end{aligned}
    \label{eq:6eq}
\end{equation}
where $\alpha_k$ is the volume fraction of phase $k \in \{1,2\}$, $\rho_k$ is its density, $\ubf_k$ its velocity, $p_k$ its pressure and $E_k$ its energy. The source terms in the momentum and energy equations include the contribution of external forces, drag forces and heat exchanges. These models are closed by the definitions of total energies,
\begin{equation}
    E_k=\frac{1}{2} \norm{\ubf_k}^2 + e_k,
	\label{eq:total_energy}
\end{equation}
where $e_k$ is the specific internal energy of phase $k$, the equations of state
\begin{equation}
    \begin{aligned}
        p_k&=p_k(\rho_k,\rho_k e_k),\\
        \theta_k&=\theta_k(\rho_k,\rho_k e_k)
    \end{aligned}
    \label{eq:eos}
\end{equation} where $\theta_k$ is the temperature of phase $k$ (which may appear in $S_k^e$), and the equality of pressures 
\begin{equation}
    p_1(\rho_1,\rho_1e_1)=p_2(\rho_2,\rho_2e_2)
	\label{eq:eq_p}
\end{equation}
which permits to deduce the void fractions since $\alpha_1 + \alpha_2 = 1$.	

It is well-known that this model is not well-posed in the sense of Hadamard, since the convective part entails complex eigenvalues~\cite{ramshaw_characteristics_1978, stewart_stability_1979}. Several directions have been explored to stabilize this model, like the addition of viscous phenomena, surface tension~\cite{ramshaw_characteristics_1978}, virtual mass~\cite{bestion_physical_1990,ndjinga_influence_2007} or internal capillary forces~\cite{bresch_global_2010}. It is also possible to consider two-velocity and two-pressure models~\cite{baer_two-phase_1986,ransom_hyperbolic_1984}, which are well-posed for smooth solutions~\cite{coquel_two-properties_2014}.

\begin{figure}[!ht]
    \centering
    \begin{tikzpicture}[scale=0.47]
        \begin{scope}[canvas is zx plane at y=3]
            \draw (-2,-8) rectangle ++(4,16);
        \end{scope}
        \begin{scope}[canvas is zx plane at y=-3]
            \draw (-2,-8) rectangle ++(4,16);
        \end{scope}


        \path[fill=red!30!white](-8,0,-2) .. controls +(2,1,0) and +(-2,-.8,0).. (0,0,-2)
         .. controls +(2,1,0) and +(-2,-.6,0) .. (8,0,-2)
         .. controls +(0,-.8,2) and +(0,1.5,0.01) ..  (8,0,2)
         ..controls +(-2,-.6,0) and +(2,1,0) ..(0,0,2)
         .. controls +(-2,-.8,0) and +(2,1,0) ..(-8,0,2)
         ..controls +(0,1.5,0.01) and +(0,-.8,2) ..(-8,0,-2);
 
         \draw[thick, color=red](-8,0,-2) .. controls +(2,1,0) and +(-2,-.8,0).. (0,0,-2)
         .. controls +(2,1,0) and +(-2,-.6,0) .. (8,0,-2)
         .. controls +(0,-.8,2) and +(0,1.5,0.01) ..  (8,0,2)
         ..controls +(-2,-.6,0) and +(2,1,0) ..(0,0,2)
         .. controls +(-2,-.8,0) and +(2,1,0) ..(-8,0,2)
         ..controls +(0,1.5,0.01) and +(0,-.8,2) ..(-8,0,-2);
 

        \draw [thick,<->] (-8,3.5,-2) -- (8,3.5,-2) ;
        \draw [thick,<->]  (-8.5,-3,2) -- (-8.5,3,2) ; 
        \draw[thick,<->] (-8.3,3,1.8)-- (-8.3,3,-2.3);
        \draw (0,3.5,-2) node[above] {$L$}; 
        \draw (-9,0,2) node {$D$};
        \draw (-8.3,3.1,-0.15) node[left] {$L$};

        \draw[thick, ->] (-9,-3.5,2.5)--(-8,-3.5,2.5);
        \draw (-8,-3.5,2.5) node[right] {$x$};
        \draw[thick, ->] (-9,-3.5,2.5)--(-9,-2.5,2.5);
        \draw (-9,-2.5,2.5) node[above] {$z$};
        \draw[thick, ->] (-9,-3.5,2.5)--(-9,-3.5,3.5);
        \draw (-9,-3.5,4) node {$y$};

        \draw (-7,1.5,0) node {Phase $2$};
        \draw (-7,-1.5,0) node {Phase $1$};
    \end{tikzpicture}
\caption{Flow configuration}
\label{fig:intro_config_ecoulement}
\end{figure}
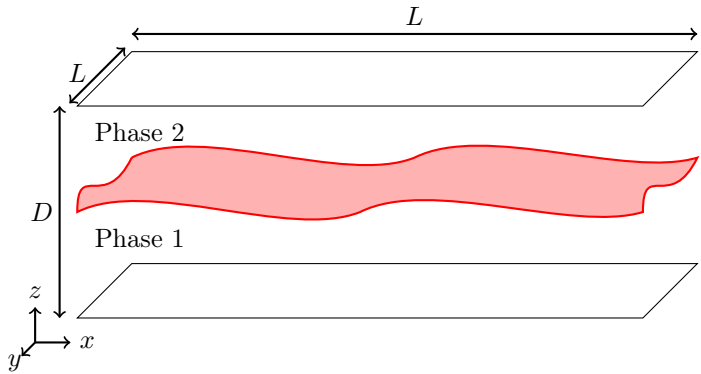

Our work carries a derivation of model~\cref{eq:6eq}-\cref{eq:eq_p}. We start from a thin three-dimensional domain: the vertical length is assumed to be much smaller than the horizontal lengths, the ratio being described by the thickness parameter $\varepsilon=D/L \ll 1$ (see~\cref{fig:intro_config_ecoulement}). The standard averaging process is replaced by an integration with respect to the vertical variable, resulting in a two-dimensional averaged model. In order to obtain a two-velocity model, our derivation circumvents the smoothness of the velocity field by making two complementary choices on the three-dimensional model. First, instead of imposing the continuity of the velocity field through the interfaces between phases, we allow a tangential friction at the interface by prescribing a Navier condition on the relative tangential velocity. Second, in order to concentrate the efforts on the interaction of this interface condition with the averaging process, we study a simplifed configuration: we assume an initial and persistent stratified configuration, in the sense that one fluid remains below the other, and the position of the interface can be parametrized by a function which only depends on the horizontal space variables $(x,y)$. We finally set walls at the bottom and at the top. After rescaling the equations, we perform an asymptotic expansion of the microscopic model with respect to $\varepsilon$ and average the equations to obtain the two-dimensional model~\cref{eq:6eq}-\cref{eq:eq_p} at the principal order. Thanks to this particular setting, it is possible to understand the influence of the different parameters of the initial model on the averaged equations, with respect to the thickness parameter. Compared to the works of Stewart and Wendroff and Ransom and Hicks~\cite{stewart_two-phase_1984,ransom_hyperbolic_1984}, the introduction of Navier conditions and the asymptotic analysis enable the deduction of a two-velocity averaged model without additional heuristic arguments. Moreover, our framework enables to deduce interaction source terms. Since this framework is intrinsically multidimensional, the extension of mathematical homogenization described above seems out of reach for the moment.

This paper is divided into three parts. \Cref{sec:barotropic-flows} introduces the configuration of the flow and the three-dimensional microscopic model in the barotropic case. \Cref{sec:asymptotic_model_barotropic} presents the averaged two-velocity, one-pressure model derived from the equations of the first section and develops the proof of this result. \Cref{sec:Navier-Stokes-Fourier} is dedicated to the extension to the non-barotropic case, therefore adding two energy equations in the three-dimensional model. The resulting averaged model is described by two velocities, one pressure and two temperatures. 

\section{Barotropic two-phase stratified flows}\label{sec:barotropic-flows}

In this section, we detail the configuration of the flow, with a particular care for the description of the stratification. We then introduce the compressible barotropic Navier-Stokes systems which describe the flow. These systems are completed by Navier boundary conditions, which model the friction of the fluids on the top and bottom of the domain. We also prescribe Navier conditions on the relative velocity of the fluids at the interface, which depict the friction between both fluids tangentially to the interface. Finally, we develop the different scalings which take place in our configuration and deduce a dimensionless version of the three-dimensional viscous model. 

\subsection{Stratified configuration and definition of the interface}

We consider a stratified two-phase flow in a thin domain \begin{equation}
    \Omega_D =\{ (\xbf,z)=(x,y,z) \in \T^2_L \times (0,D)\}, \quad L, \, D>0
\label{1_def_domain}
\end{equation} where $\T^2_L$ denotes the two-dimensional torus of period $L$.

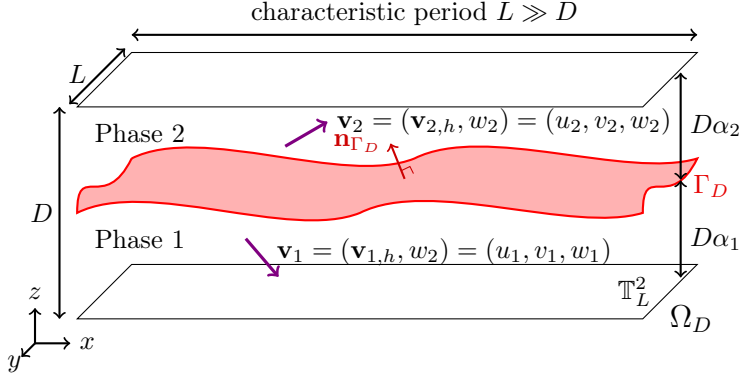
\begin{figure}[!ht]\label{fig1: config ecoulement}
    \centering
    \begin{tikzpicture}[scale=0.47]
        \begin{scope}[canvas is zx plane at y=3]
            \draw (-2,-8) rectangle ++(4,16);
        \end{scope}
        \begin{scope}[canvas is zx plane at y=-3]
            \draw (-2,-8) rectangle ++(4,16);
            \draw (0,7) node {$\T_L^2$};
        \end{scope}
        \draw (8.5,-3,2) node[right]{\large $\Omega_D$};


        \path[fill=red!30!white](-8,0,-2) .. controls +(2,1,0) and +(-2,-.8,0).. (0,0,-2)
         .. controls +(2,1,0) and +(-2,-.6,0) .. (8,0,-2)
         .. controls +(0,-.8,2) and +(0,1.5,0.01) ..  (8,0,2)
         ..controls +(-2,-.6,0) and +(2,1,0) ..(0,0,2)
         .. controls +(-2,-.8,0) and +(2,1,0) ..(-8,0,2)
         ..controls +(0,1.5,0.01) and +(0,-.8,2) ..(-8,0,-2);
 
         \draw[thick, color=red](-8,0,-2) .. controls +(2,1,0) and +(-2,-.8,0).. (0,0,-2)
         .. controls +(2,1,0) and +(-2,-.6,0) .. (8,0,-2)
         .. controls +(0,-.8,2) and +(0,1.5,0.01) ..  (8,0,2)
         ..controls +(-2,-.6,0) and +(2,1,0) ..(0,0,2)
         .. controls +(-2,-.8,0) and +(2,1,0) ..(-8,0,2)
         ..controls +(0,1.5,0.01) and +(0,-.8,2) ..(-8,0,-2);
 
        \draw [thick,->,color=red!80!black] (0.3,0,-0.5) -- (-0.1,1, -0.5) ; 
        \draw [color=red!80!black] (-0.1,1,-0.5) node[left] {$\nbf_{\Gamma_D}$};
        \draw [color=red!90!black] (8.4,0,0) node[right] {$\Gamma_D$};
        \draw [color=red!80!black] (0.2,0.23,-0.5)--(0.45,0.35,-0.5)--(0.55,0.15,-0.5);

        \draw [thick,<->] (-8,3.5,-2) -- (8,3.5,-2) ;
        \draw [thick,<->]  (-8.5,-3,2) -- (-8.5,3,2) ; 
        \draw[thick,<->] (-8.3,3,1.8)-- (-8.3,3,-2.3);
        \draw (0,3.5,-2) node[above] {characteristic period $L\gg D$}; 
        \draw (-9,0,2) node {$D$};
        \draw (-8.3,3.1,-0.15) node[left] {$L$};

        \draw[thick, ->] (-9,-3.5,2.5)--(-8,-3.5,2.5);
        \draw (-8,-3.5,2.5) node[right] {$x$};
        \draw[thick, ->] (-9,-3.5,2.5)--(-9,-2.5,2.5);
        \draw (-9,-2.5,2.5) node[above] {$z$};
        \draw[thick, ->] (-9,-3.5,2.5)--(-9,-3.5,3.5);
        \draw (-9,-3.5,4) node {$y$};

        \draw (-7,1.5,0) node {Phase $2$};
        \draw (-7,-1.5,0) node {Phase $1$};

        \draw [->,color=blue!50!red,very thick] (-2.5,1.5,1) -- (-1.5,2,0.5) ;
        \draw (-1.5,2,0.5) node[right] {$\vbf_2=(\vbf_{2,h},w_2)=(u_2,v_2,w_2)$};
        \draw [->,color=blue!50!red,very thick] (-4,-1.5,0) -- (-2.5,-2,1.5);
        \draw (-3,-1.5,1) node [right] {$\vbf_1=(\vbf_{1,h},w_2)=(u_1,v_1,w_1)$};

        \draw[thick,<->] (8.3,3.2,0)--(8.3,0.17,0);
        \draw (8.3,1.7,0) node[right] {$D\alpha_2$};
        \draw[thick,<->] (8.3,0.15,0)--(8.3,-2.6,0);
        \draw (8.3,-1.5,0) node[right] {$D\alpha_1$};
    \end{tikzpicture}
    \caption{Flow configuration}
\end{figure}

The model we study represents a compressible two-phase flow where the fluids are supposed immiscible. The domain filled by each fluid is defined thanks to the top and bottom boundary and the interface between each fluid. We assume that the interface $\Gamma_D$ can be described by the use of continuous function $\alpha_1 : \R_+ \times \R_L^2 \to (0,1)$ in the following way:
\begin{equation}
    \forall t \in \R_+, \qquad \Gamma_D(t)=\{ (\xbf,D\alpha_1(t,\xbf)), \; \xbf \in \T_L^2\}.
    \label{1 def interface}
\end{equation}
The existence of $\alpha_1$ ensures that the stratified configuration is preserved through time. Let us define a velocity field $\vbf_i:\R_+ \times \R^2 \rightarrow \R^3$ which describes the velocity of particles located at the interface. The displacement of the latter is linked to the component of $\vbf_i$ which is normal to the interface thanks to the kinematic equation
\begin{equation}
    \partial_t (D\alpha_1) - \vbf_i \cdot \Big(\sqrt{\Norm{\nabla_h(D\alpha_1)}^2+1}\Big) \nbf_{\Gamma_D}=0 \quad \text{in } \R_+ \times \T_L^2,
\label{1_fraction}
\end{equation}

\begin{flalign*}
    \text{where} && \Big(\sqrt{\Norm{\nabla_h(D\alpha_1)}^2+1}\Big) \nbf_{\Gamma_D}&= \trans{\begin{pmatrix}
    -\nabla_h(D\alpha_1) & 1
    \end{pmatrix}}&&
\end{flalign*}
is the unitary normal vector to the interface pointing upwards (see Figure~\cref{fig1: config ecoulement}) and $\nabla_h$ is the horizontal gradient operator $\begin{pmatrix}
\partial_x & \partial_y
\end{pmatrix}^T$. 

The domain $\Omega_{D,1}(t)$ filled by the first phase is described by the characteristic function 
\begin{equation*}
    \forall \xbf \in \T_L^2, \qquad \chi_{D,1}(t,\xbf,z)=\left\{\begin{array}{rl}
    1 & \text{for } 0< z < D \alpha_1(t,\xbf),\\
    0 & \text{otherwise.}
    \end{array}\right.
\end{equation*}

\noindent Similarly, the domain $\Omega_{D,2}(t)$ filled by the second phase is described by the characteristic function $\chi_{D,2}=1-\chi_{D,1}$. 

\subsection{The compressible barotropic stratified model}

Let $k \in \{1,2\}$. Fluid $k$ is described by its density $\rho_k: \R_+^* \times \Omega_{D,k} \rightarrow \R_+$, its velocity $\vbf_k=(u_k,v_k,w_k) : \R_+^* \times \Omega_{D,k} \rightarrow \R^3$ and its pressure $p_k : \R_+^* \times \Omega_{D,k} \rightarrow \R$. Since we consider viscous flows, we denote $\mu_k$ and $\lambda_k$ the shear and bulk viscosity of the fluid. This flow is described by two compressible, barotropic and viscous Navier-Stokes systems --- one for each fluid.
Let $h_0=0$, $h_1=\alpha_1$ and $h_2=1$. We consider the following system for \[ k\in \{ 1,2\}, \quad t>0, \quad (\xbf,z) \in \Omega_{D,k}(t)=\T_L^2\times (D h_{k-1}(t,\xbf),D h_k(t,\xbf)):\]

\begin{align}
    \partial_t \rho_k + \nabla \cdot (\rho_k \vbf_k)&=0,\label{1_NS_cm}\\
    \partial_t (\rho_k \vbf_k) + \nabla \cdot (\rho_k \vbf_k \otimes \vbf_k) + \nabla p_k &= \nabla \cdot \S_k,\label{1_NS_qdm}\\
    p_k&=p_k(\rho_k)\label{1_barotropy_laws}
\end{align}
where the viscous tensor $\S_k : \R_+^* \times \Omega_{D,k} \rightarrow \mcal_3(\R)$ associated with fluid $k$ is of the form 
\begin{equation}
    \S_k = 2\mu_k \ubrace{\D_k}{\frac{(\nabla \otimes \vbf_k + {(\nabla \otimes \vbf_k)}^T)}{2}} + \lambda_k (\nabla \cdot \vbf_k) \I
\label{1_def_viscous_tensor}
\end{equation} 
and the equations of state~\cref{1_barotropy_laws} are regular and increasing. Equation~\cref{1_NS_cm} illustrates the conservation of mass. The next one is the momentum equation, which is obtained thanks to Newton's second principle.

\begin{remark}
    In the equations above, we did not include any effects of gravity or surface tension. Their effects may help to prove the preservation of the stratified configuration and generally influence the stability of the configuration~\cite{ramshaw_characteristics_1978}. Moreover, we assume in all this work that the systems we study admit strong solutions.  
\end{remark}

To complete the system description, boundary and interface conditions are necessary. Denoting $\vbf_{k,h}$ the horizontal component of $\vbf_k$ for $k \in \{1,2\}$, and
\[ \S_k = \begin{pmatrix}
\S_{k,h}\\
\S_{k,z}
\end{pmatrix}=\left(\begin{array}{c|c}
\S_{k,hh} & \S_{k,hz}\\ \hline
\S_{k,zh} & \S_{k,zz}
\end{array}\right), \text{ where } \S_{k,hh} \in \mcal_{2,2}(\R),\]
we prescribe Navier boundary conditions on the top and bottom of the domain: 
\begin{equation}
    \left\{ \begin{aligned}
    w_1|_{z=0} &=& 0,\\
    -\S_{1,hz} +\kappa_1 \vbf_{1,h} |_{z=0}&=&0,
    \end{aligned}\right.
\label{1_Navier_bottom}
\end{equation}

\begin{equation}
    \left\{ \begin{aligned}
    w_2|_{z=D} &=& 0,\\
    \S_{2,hz} + \kappa_2 \vbf_{2,h} |_{z=D}&=&0.
    \end{aligned}\right.
\label{1_Navier_top}
\end{equation}
The first equation of each boundary condition depicts the impermeability of the domain. The second one describes the horizontal friction of the fluid on the boundary. The quantities $\kappa_1$ and $\kappa_2$ are constant friction coefficients at the bottom and top boundary respectively. 

\medskip 

There remains to prescribe interface conditions. Defining the total stress tensor 
\begin{equation}
    \sigma_k = (-p_k+\lambda_k (\nabla \cdot \vbf_k))\I + 2\mu_k \D_k=-p_k \I + \S_k,
\label{1_def_total_stress_tensor}
\end{equation}
we prescribe the continuity of its normal component to the interface:
\begin{equation}
    \sigma_1 \nbf_{\Gamma_D} |_{z=D\alpha_1} = \sigma_2 \nbf_{\Gamma_D} |_{z=D\alpha_1}.
\label{1_tensor_interface}
\end{equation}
In addition to the previous condition, we prescribe Navier conditions at the interface. They read
\begin{equation}
    \left\{ \begin{aligned}
    \forall k \in \{1,2 \}, \qquad (\vbf_k-\vbf_i) \cdot \nbf_{\Gamma_D}|_{z=D\alpha_1} &= 0\\
    {(\S_1 \nbf_{\Gamma_D} + \kappa_i (\vbf_1-\vbf_2))}_{\rm tan}|_{z=D\alpha_1}&=0,
    \end{aligned}\right.
\label{1_Navier_interface}
\end{equation}
where $X_{\rm tan}|_{z=D\alpha_1}$ denotes the component of the vector $X$ which is tangential to the interface and $\kappa_i$ is the friction coefficient at the interface. The first equation expresses the continuity of the velocity in the normal direction to the interface. The second one describes the friction of the fluids tangentially to the interface.

The orthogonal component of the Navier interface condition~\cref{1_Navier_interface} enables us to rewrite the fraction equation~\cref{1_fraction} as
\begin{equation}
    \partial_t(D\alpha_1)+\vbf_{1,h}(D\alpha_1) \cdot \nabla_h (D \alpha_1)=w_1(D\alpha_1).
\label{1_fraction_advected_form}
\end{equation}

\begin{remark}\label{remark 1: tangential stress}
    There is no need of any piece of information about the tangential component of $\vbf_i$, the normal component being the only useful one in this paper (see~\cref{1_fraction}). Looking at the tangential component of equation~\cref{1_tensor_interface}, using the fact that ${((-p_k+\lambda_k (\nabla \cdot \vbf_k))\nbf_{\Gamma_D})}_{\rm tan}|_{z=D\alpha_1}=0$ for $k\in \{1,2\}$, we have 
    \begin{equation}
        {(\S_1 \nbf_{\Gamma_D})}_{\rm tan} |_{z=D\alpha_1} = {(\S_2 \nbf_{\Gamma_D})}_{\rm tan} |_{z=D\alpha_1}.
    \label{1_tangential_stress_symmetry}
    \end{equation}
    We can then rewrite the second equation of the interface condition~\cref{1_Navier_interface} with $\S_2$ instead of $\S_1$.
\end{remark}

\begin{remark}
    We do not introduce different notations for the scalar products in two and three dimensions. The context should make the meaning of the scalar product clear.
\end{remark}

\begin{remark}\label{rmk:about_interface_conditions}
    The choice of the interface condition~\cref{1_Navier_interface} has a lot of implications on the model that is derived thanks to dimension reduction. For example, an alternative possibility is the prescription of a usual Dirichlet condition with respect to the relative velocity, i.e the continuity of the velocity field at the interface:
    \[ \vbf_1(D\alpha_1)=\vbf_2(D\alpha_1).\]
\end{remark}

\begin{remark}
    The domain occupied by each fluid being advected by the flow, we have the equation
    \begin{equation}
        \forall k \in \{1,2\}, \qquad \partial_t \chi_{D,k} + \vbf_k \cdot \nabla \chi_{D,k}= 0.
    \label{1_advection_domain}
    \end{equation}
    Thanks to the Navier interfacial condition, the previous equation is compatible with the fraction equation~\cref{1_fraction}. Indeed, integrating equation~\cref{1_advection_domain} for the first fluid over $\Omega_D$ yields
    \[ \int_{\Omega_D} (\partial_t \chi_{D,1} + \vbf_1 \cdot \nabla \chi_{D,1}) =0.\] 
    \begin{flalign*}
        \text{Firstly,} && \int_{\Omega_D} \partial_t \chi_{D,1}&=\int_{\T_L^2}\partial_t \int_{(0,D)} \chi_{D,1} = \int_{\T_L^2} \partial_t (D\alpha_1)&& \\
        \text{Secondly,} && \int_{\Omega_D} \vbf_1 \cdot \nabla \chi_{D,1} &= \int_{\Omega_D} \div(\vbf_1 \chi_{D,1}) - \int_{\Omega_{D,1}} \div(\vbf_1).&&
    \end{flalign*}
    Thanks to the divergence theorem and the impermeability condition~\cref{1_Navier_bottom},
    \[\int_{\Omega_D} \vbf_1 \cdot \nabla \chi_{D,1} = - \int_{\Gamma_D} \vbf_1\cdot \nbf_{\Gamma_D} \d S .\]
    Since the surface $\Gamma_D$ is parametrized by the function $\varphi(x,y)=(x,y,D\alpha_1(x,y))$, a change of variables gives
    \[ \int_{\Gamma_D} \vbf_1\cdot \nbf_{\Gamma_D} \d S = \int_{\T_L^2} (\vbf_{1,h}(D\alpha_1)\cdot \nabla_h (D\alpha_1) -w_1(D\alpha_1))(x,y) \d x \d y =0.\]
    The normal component of the Navier interface condition~\cref{1_Navier_interface} then gives the compatibility between equations~\cref{1_advection_domain} and~\cref{1_fraction_advected_form}.
    The reasoning is the same for the second fluid.
\end{remark}

\subsection{Rescaling the equations in the thin domain assumption}

Let $\varepsilon=D/L$ be a small parameter corresponding to the thin domain assumption. We then introduce dimensionless variables $\tilde{x}=x/L$, $\tilde{y}=y/L$ and $\tilde{z}=z/D=z/(\varepsilon L)$. Denoting $T$ the characteristic time of the flow, we also introduce $\tilde{t}/T$. Denoting $\T^2=\T_1^2$, switching the variables allows to study the flow on the rescaled domain 
\begin{equation}
    \Omega = \T^2 \times (0,1).
\end{equation}
In the same spirit, we introduce characteristic dimensions for every unknown in the system: $U=L/T$ for the horizontal velocity, $W=\varepsilon U$ for the vertical velocity, $R$ for the density, $P=RU^2$ for the pressure, $M=RUL$ for the viscosity, and $K=RU$ for the friction coefficients. Denote the dimensionless unknowns $\tilde{u}=u/U$, $\tilde{w}=w/W$, $\tilde{\rho}=\rho/R$, $\tilde{p}=p/P$. Denote $\tilde{\mu}=\mu/M$ the inverse of the Reynolds number, $\tilde{\lambda}=\lambda/M$ and $\tilde{\kappa}=\kappa/K$. We state that the dimensionless viscosity and friction coefficients are constant parameters which may depend on $\varepsilon$. Moreover, we assume that their scaling is independent of the fluid they characterize.

Denote $\chi_k$ the rescaled version of $\chi_{D,k}$ for $k \in \{1,2\}$, $\Omega_k$ the rescaled version of $\Omega_{D,k}$ and $\Gamma$ the rescaled version of $\Gamma_D$. Removing the $\tilde{\cdot}$ notation for the sake of clarity, and denoting $\Delta_h =\nabla_h \cdot \nabla_h$, we rewrite the equations~\cref{1_NS_cm}-\cref{1_advection_domain} as follows: 

\begin{lemma}\label{Lemma 1}
    The rescaling of model~\cref{1_NS_cm}-\cref{1_advection_domain} results in the following set of equations on the rescaled domain $\Omega$.

    \begin{itemize}
        \item \textbf{Dimensionless advection of the domain}
        \begin{equation}
            \partial_t \chi_k + \vbf_k \cdot \nabla \chi_k =0.
        \label{1_rescaled_advection_domain}
        \end{equation}
        \item \textbf{Dimensionless fraction equation}
        \begin{equation}
            \partial_t \alpha_1 + \vbf_{1,h}(\alpha_1) \cdot \nabla_h \alpha_1 = w_1 (\alpha_1).
        \label{1_rescaled_fraction}
        \end{equation}

        \item \textbf{Dimensionless mass conservation equations}
        \begin{equation}
            \forall k \in \{1,2\}, \qquad \partial_t \rho_k  + \nabla_h \cdot (\rho_k  \vbf_{k,h} )+\partial_z (\rho_k  w_k )=0,
        \label{1_rescaled_cm}
        \end{equation}

        \item \textbf{Dimensionless horizontal momentum equations}
        \begin{multline}
            \forall k \in \{1,2\}, \quad  \partial_t (\rho_k  \vbf_{k,h} ) + \nabla_h \cdot (\rho_k  \vbf_{k,h}  \otimes \vbf_{k,h} )+ \partial_z(\rho_k  w_k  \vbf_{k,h} )+ \nabla_h p_k \\
            =(\lambda_k+\mu_k)\nabla_h(\nabla \cdot \vbf_{k} )
            + \mu_k \Delta_h \vbf_{k,h}  
            + \frac{\mu_k}{\varepsilon^2} \partial_{zz}\vbf_{k,h} ,
        \label{1_rescaled_qdm_h}
        \end{multline}

        \item \textbf{Dimensionless vertical momentum equations}
        \begin{multline}
            \varepsilon(\partial_t (\rho_k  w_k )+\nabla_h \cdot (\rho_k  \vbf_{k,h}w_k)+\partial_z (\rho_k  {(w_k )}^2))+\frac{1}{\varepsilon}\partial_z p_k  \\
            = \frac{1}{\varepsilon}((\lambda_k+\mu_k)\partial_z (\nabla  \cdot \vbf_k )+ \mu_k\partial_{zz} w_k )+\mu_k\varepsilon \Delta_h w_k .
        \label{1_rescaled_qdm_z}
        \end{multline}

        \item \textbf{Dimensionless Navier boundary conditions}
        \begin{equation}
            \left\{ \begin{aligned}
            w_1 |_{z=0} &=& 0,\\
            -\frac{\mu_1}{\varepsilon}\partial_z \vbf_{1,h} +\kappa_1 \vbf_{1,h}  |_{z=0}&=&0,
            \end{aligned}\right.
        \label{1_rescaled_bottom}
        \end{equation}

        \begin{equation}
            \left\{ \begin{aligned}
            w_2 |_{z=1} &=& 0,\\
            \quad \frac{\mu_2}{\varepsilon}\partial_z \vbf_{2,h} 
            + \kappa_2 \vbf_{2,h}  |_{z=1}&=&0.
            \end{aligned}\right.
        \label{1_rescaled_top}
        \end{equation}

        \item \textbf{Dimensionless continuity of the normal stress at the interface}
        \begin{align}
        \begin{split}
            \Big( \partial_x \alpha_1 (p_1 -\lambda_1\nabla \cdot \vbf_1 -2\mu_1 \partial_x u_1 ) -\mu_1 \partial_y \alpha_1 (\partial_y u_1 +\partial_x v_1 ) &\\
            +\mu_1 \partial_x w_1  + \frac{\mu_1}{\varepsilon^2} \partial_z u_1 \Big)\Big|_{z=\alpha_1}&={(\cdots)}_2\big|_{z=\alpha_1},
        \end{split}
        \label{1_rescaled_tensor_x}\\
        \begin{split}
            \Big( \partial_y \alpha_1 (p_1 -\lambda_1\nabla \cdot \vbf_1 -2\mu_1 \partial_y v_1 )-\mu_1 \partial_x \alpha_1 (\partial_y u_1 +\partial_x v_1 )&\\
            +\mu_1 \partial_y w_1  + \frac{\mu_1}{\varepsilon^2} \partial_z v_1 \Big)\Big|_{z=\alpha_1}&={(\cdots)}_2\big|_{z=\alpha_1},
        \end{split}
        \label{1_rescaled_tensor_y}\\
        \begin{split}
            \Big( -p_1 +\lambda_1\nabla \cdot \vbf_1 +2 \mu_1\partial_z w_1  - \mu_1 \partial_x \alpha_1(\partial_z u_1  + \varepsilon^2 \partial_x w_1 )&\\
            -\mu_1 \partial_y \alpha_1(\partial_z v_1 +\varepsilon^2 \partial_y w_1 )\Big)\Big|_{z=\alpha_1}
            &={(\cdots)}_2\big|_{z=\alpha_1},
        \end{split}
        \label{1_rescaled_tensor_z}
        \end{align}
        where ${(\cdots)}_2$ denotes the same expression as the left-hand side of every condition with the subscript $1$ replaced by $2$, except for $\alpha_1$.

        \item \textbf{Dimensionless normal Navier interface condition}
        \begin{equation}
            \Big(-u_1 \partial_x \alpha_1-v_1  \partial_y \alpha_1+w_1 \Big)\Big|_{z=\alpha_1} = \Big(-u_2 \partial_x \alpha_1-v_2  \partial_y \alpha_1+w_2 \Big)\Big|_{z=\alpha_1}
        \label{1_rescaled_navier_orth}
        \end{equation}

        \item \textbf{Dimensionless tangential Navier interface condition}
        \begin{align}
        \begin{split}
            \mu_1 \varepsilon \Big( \frac{1}{\varepsilon^2}\partial_z u_1  + \partial_x w_1  -2 \partial_x \alpha_1(\partial_x u_1 -\partial_z w_1 ) - \partial_y \alpha_1 (\partial_y u_1 +\partial_x v_1 ) &\\
            - \partial_x \alpha_1 \partial_y \alpha_1 (\partial_z v_1  + \varepsilon^2 \partial_y w_1 ) - {(\partial_x \alpha_1)}^2 (\partial_z u_1 +\varepsilon^2 \partial_x w_1 )\Big)& \\
            + \kappa_i (u_1 -u_2 +\varepsilon^2 \partial_x \alpha_1 (w_1 -w_2 ))\Big|_{z=\alpha_1} &=0,
        \end{split} 
        \label{1_rescaled_navier_tan_x}\\
        \begin{split}
            \mu_1 \varepsilon \Big( \frac{1}{\varepsilon^2}\partial_z v_1  + \partial_y w_1  -2 \partial_y \alpha_1(\partial_y v_1 -\partial_z w_1 ) - \partial_x \alpha_1 (\partial_x v_1 +\partial_y u_1 ) &\\
            - \partial_y \alpha_1 \partial_x \alpha_1 (\partial_z u_1  + \varepsilon^2 \partial_x w_1 ) - {(\partial_y \alpha_1)}^2 (\partial_z v_1 +\varepsilon^2 \partial_y w_1 )\Big)& \\
            + \kappa_i (v_1 -v_2 +\varepsilon^2 \partial_y \alpha_1 (w_1 -w_2 ))\Big|_{z=\alpha_1} &=0.
        \end{split}
        \label{1_rescaled_navier_tan_y}
        \end{align}
    \end{itemize}
\end{lemma}

\begin{proof}
    The introduction of the dimensionless variables and unknowns in the various equations, boundary conditions and interface conditions is performed using the chain rule. The characteristic function $\chi_k$ is advected by the dimensionless velocity, hence equation~\cref{1_rescaled_advection_domain}. The rescaling procedure immediately yields the fraction equation~\cref{1_rescaled_fraction}, the mass conservation equation~\cref{1_rescaled_cm} and the momentum equation, which we split into two components --- the horizontal component~\cref{1_rescaled_qdm_h} and the vertical component~\cref{1_rescaled_qdm_z}.

    Rescaling the boundary conditions yields equations~\cref{1_rescaled_bottom} and~\cref{1_rescaled_top}. Since the normal vector to the interface is now 
    \begin{equation}
        \nbf_{\Gamma}=\frac{1}{\sqrt{1+\varepsilon^2 |\nabla_h \alpha_1|^2}}
        \begin{pmatrix}
        -\varepsilon \nabla_h \alpha_1\\
        1
        \end{pmatrix},
    \label{1_rescaled_normal}
    \end{equation}
    the rescaling of equation~\cref{1_tensor_interface} detailed into every component gives~\cref{1_rescaled_tensor_x},~\cref{1_rescaled_tensor_y} and~\cref{1_rescaled_tensor_z}.

    Now, the Navier interface conditions must be rescaled. Using~\cref{1_rescaled_normal}, we can rescale the continuity of the velocity orthogonally to the interface and thus obtain~\cref{1_rescaled_navier_orth}. 
    The tangential condition requires more computations. First, let us introduce two tangential vectors to the interface: 
    \begin{equation*}
        \tbf_{\Gamma_D}^x = \begin{pmatrix}
        1\\0\\ \partial_x (D\alpha_1)
        \end{pmatrix} \quad \text{and} \quad \tbf_{\Gamma_D}^y = \begin{pmatrix}
        0\\1\\ \partial_y (D\alpha_1)
        \end{pmatrix}
    \end{equation*}
    \noindent Those two are rescaled as
    \[ \tbf_\Gamma^x = \begin{pmatrix}
    1\\0\\ \varepsilon \partial_x \alpha_1
    \end{pmatrix} \quad \text{and} \quad \tbf_{\Gamma}^y = \begin{pmatrix}
    0\\1\\ \varepsilon \partial_y \alpha_1
    \end{pmatrix}.\]
    \noindent Then, the tangential component of condition~\cref{1_Navier_interface} is equivalent to the system
    \begin{equation*}
        \begin{split}
            (\S_1 \nbf_{\Gamma_D} + \kappa_i (\vbf_1-\vbf_2))\cdot \tbf_{\Gamma_D}^x (D\alpha_1)&=0,\\
            (\S_1 \nbf_{\Gamma_D} + \kappa_i (\vbf_1-\vbf_2))\cdot \tbf_{\Gamma_D}^y (D\alpha_1)&=0.
        \end{split}
    \end{equation*}
    which is rescaled as~\cref{1_rescaled_navier_tan_x}-\cref{1_rescaled_navier_tan_y}.
\end{proof}

\section{The asymptotic system in the barotropic case}\label{sec:asymptotic_model_barotropic}

In this section, we present the asymptotic model derivated from the compressible barotropic double Navier-Stokes system~\cref{1_NS_cm}-\cref{1_fraction_advected_form}. The first subsection introduces the fundamental notion of average that will be used later, and states the theorem presenting the asymptotic model. In a second part, we develop the derivation process.
\subsection{Presentation of the averaged model}

Before presenting our asymptotic model, we need to introduce some notations. As explained earlier, integrating the Navier-Stokes equations causes the appearance of averaged unknowns. Since the majority of the terms in the Navier-Stokes equations are nonlinear in the compressible setting and since the multiplication and the integration do not commute, it is necessary to define different notions of averages.

\begin{definition}[Averages]
    Let $f$ be a function defined on the domain filled by fluid $k \in \{1,2\}$. The average of $f$, which we denote $\bar{f}$, is defined by
    \begin{equation}
        \forall (t,\xbf) \in \R_+ \times \T^2, \qquad \bar{f}(t,\xbf)=\frac{1}{\alpha_k}\int_{h_{k-1}}^{h_k} f(t,\xbf,z) \d z.
    \end{equation}
    The Favre average of $f$ with respect to $\rho_k$, which we denote $\favre{f}$, is defined for all $(t,\xbf) \in \R_+ \times \T^2$ by 
    \begin{equation}
        \favre{f}(t,\xbf)=\frac{\bar{\rho_k f}}{\bar{\rho_k}} = \frac{1}{\alpha_k \bar{\rho_k}}\int_{h_{k-1}}^{h_k} \rho_k f(t,\xbf,z) \d z.
    \end{equation}
\end{definition}

\begin{theorem}\label{thm_MR}
    Assume the following scalings for viscosity and friction coefficients:
    \begin{equation}
        \forall k\in \{1,2\}, \qquad \left\{ \begin{array}{rll}
        \mu_k=\widehat{\mu_k}\varepsilon^\tau, &\lambda_k=\widehat{\lambda_k}\varepsilon^\tau, & 0<\tau < 2 ; \\
        \kappa_{i}=\widehat{\kappa_i} \varepsilon, & \kappa_k=\widehat{\kappa_k}\varepsilon^\xi, & \xi\geqslant 1, 
        \end{array}\right.
    \label{MR_profile_viscosity_friction}
    \end{equation}
    where the coefficients $\widehat{\mu_k}$, $\widehat{\lambda_k}$, $\widehat{\kappa_i}$ and $\widehat{\kappa_k}$ are positive constants independent of $\varepsilon$.

    \noindent Define
    \[ \delta_\xi = \left\{ \begin{array}{rl}
        1 & \text{if } \xi =1,\\
        0 & \text{if } \xi >1.
    \end{array}\right.\]

    \noindent Let $(\alpha_1,\rho_1,\vbf_1,p_1,\alpha_2,\rho_2,\vbf_{2},p_2)$ be a strong solution to the Navier-Stokes system with Navier boundary and interface conditions~\cref{1_NS_cm}-\cref{1_fraction_advected_form}.

    \noindent The functions $(\alpha_1,\bar{\rho_1},\favre{\vbf_{1,h}},\bar{p_1},\alpha_2,\bar{\rho_2},\favre{\vbf_{2,h}},\bar{p_2})$ satisfy the following closed system:
    \begin{itemize}
    \item \textbf{Averaged mass conservation equations}
    \begin{equation}
        \forall k \in \{1,2\}, \qquad \partial_t(\alpha_k \bar{\rho_k}) + \nabla_h \cdot (\alpha_k \bar{\rho_k} \favre{\vbf_{k,h}})=0.
    \label{MR-cm}
    \end{equation}

    \item \textbf{Averaged momentum equations}
    \begin{multline}
        \forall k \in \{1,2\}, \qquad \partial_t (\alpha_k \bar{\rho_k} \favre{\vbf_{k,h}}) + \nabla_h \cdot (\alpha_k \bar{\rho_k} \favre{\vbf_{k,h}} \otimes \favre{\vbf_{k,h}})+ \nabla_h(\alpha_k \bar{p_k})\\
        =p_i \nabla_h \alpha_k -\delta_\xi \widehat{\kappa_k}\favre{\vbf_{k,h}}+{(-1)}^k\widehat{\kappa_i}(\favre{\vbf_{1,h}}-\favre{\vbf_{2,h}}),
    \label{MR-qdm}
    \end{multline}
    \noindent where $p_i$ is any convex combination of the averaged pressures $\bar{p_k}$. This equation is satisfied up to an error of order $\varepsilon^\tau+\varepsilon^{2-\tau}+(1-\delta_\xi)\varepsilon^{\xi-1}$.

    \item \textbf{Averaged equations of state}
    \begin{equation}
        \forall k \in \{1,2\}, \qquad \bar{p_k}=p_k(\bar{\rho_k}).
    \label{MR-barotropy}
    \end{equation}
    This equation is satisfied up to an error of order $\varepsilon^\tau$. 

    \item \textbf{Equality of the pressures}
    \begin{equation}
        \bar{p_1}=\bar{p_2}.
    \label{MR-barotropy-equality}
    \end{equation}
    This equation is satisfied up to an error of order $\varepsilon^\tau$. 

    \item \textbf{Saturation of the domain}
    \begin{equation}
        \alpha_1+\alpha_2=1.
    \label{MR-sum-alpha}
    \end{equation}
    \end{itemize}
\end{theorem}

\begin{remark}
    The approximation errors mentioned above show the necessity of prescribing these specific scalings of viscosity and friction coefficients. Our work provides the approximation errors and the necessary conditions for the model to be valid. The exponents $\xi$ and $\tau$ are independent. 
\end{remark}

\begin{remark}
    In the tangential component of the Navier interface condition~\cref{1_Navier_interface}, we can replace the term $\kappa_i (\vbf_1-\vbf_2)$ by a quadratic one $\kappa_i |\vbf_1-\vbf_2|(\vbf_1-\vbf_2)$. The resulting term in the averaged model becomes $\kappa_i |\favre{\vbf_1}-\favre{\vbf_2}|(\favre{\vbf_1}-\favre{\vbf_2})$ and causes no change to the error approximation.
\end{remark}

\begin{remark}
    The scalings detailed in~\cref{MR_profile_viscosity_friction} result in an asymptotic model that is compressible and inviscid. In other words, our asymptotic limit corresponds to moderate Mach numbers and high Reynolds numbers.
\end{remark}

\subsection{Proof of~\texorpdfstring{\cref{thm_MR}}{thm MR}: derivation of the asymptotic model}

The proof of~\cref{thm_MR} is divided into several steps. Starting from the rescaled equations~\cref{1_rescaled_advection_domain}-\cref{1_rescaled_navier_tan_y}, we first simplify the vertical momentum equation and the interface conditions by neglecting terms of high order in $\varepsilon$. Then, we integrate the mass conservation equation and the horizontal momentum equation in the vertical direction. This procedure causes boundary terms to appear in the equations, which must be closed using the rescaled boundary and interface conditions. Finally, we prescribe specific profiles for the viscosity and friction coefficients for the asymptotic model to be valid.

\subsubsection{Simplification of the vertical momentum equation and interface conditions}\label{subsec:asym-analysis}

    Some of the equations and conditions listed in~\cref{Lemma 1} are very complex once rescaled. However, some terms are of high order in $\varepsilon$, and therefore become very small in our setting ($\varepsilon \ll 1$). Dropping these terms allows a first simplification of the system. To make sure that the derived system is really an approximation of the initial double Navier-Stokes system, we keep trace of the errors that we cause along the way using Landau notations. Since we have assumed that the scaling of the viscosity and friction coefficients is independent of the fluid, we drop the subscript $k$ in the Landau notations.

    Performing the asymptotic analysis impacts the vertical component of the momentum equation~\cref{1_rescaled_qdm_z} as well as interface conditions~\cref{1_rescaled_tensor_z},~\cref{1_rescaled_navier_tan_x} and~\cref{1_rescaled_navier_tan_y} in the following way.

    \begin{itemize}
        \item \textbf{Asymptotic vertical momentum equation}
        \begin{equation}
            \forall k \in \{1,2\},\qquad \partial_z p_k  = (\lambda_k+\mu_k)\partial_z(\nabla \cdot \vbf_{k} )+\mu_k \partial_{zz} w_k +O(\varepsilon^2).
        \label{3_asymptotic_qdm_z}
        \end{equation}

        \item \textbf{Asymptotic continuity of the vertical component of the normal stress at the interface}
        \begin{multline}
            \Big(-p_1 +\lambda_1 \nabla \cdot \vbf_1  + 2 \mu_1\partial_z w_1  
            - \mu_1 \nabla_h \alpha_1 \cdot \partial_z \vbf_{1,h}\Big)\Big|_{z=\alpha_1}\\
            = {(\cdots)}_2\big|_{z=\alpha_1}+O(\mu \varepsilon^2).
        \label{3_asymptotic_tensor_z}
        \end{multline}

        \item \textbf{Asymptotic tangential Navier interface condition}
        \begin{equation}
            \Big(\frac{\mu_1}{\varepsilon}\partial_z \vbf_{1,h}  + \kappa_i(\vbf_{1,h} -\vbf_{2,h} )\Big)\Big|_{z=\alpha_1}=O(\mu \varepsilon+\kappa_i \varepsilon^2)
        \label{3_asymptotic_Navier_tan}
        \end{equation}
    \end{itemize}

    \begin{remark}\label{rmk_1_odg_velocity}
        In the previous paragraph, we did not simplify the horizontal momentum equation~\cref{1_rescaled_qdm_h} since the terms of order $\varepsilon^2$ are not negligible. Indeed, multiplying~\cref{1_rescaled_qdm_h} by $\varepsilon^2$ and neglecting the quadratic terms in $\varepsilon$ indeed yields 
        \begin{equation}
            \mu_k \partial_{zz} \vbf_{k,h}  = O(\varepsilon^2).
        \label{2_odg_dzz_velocity}
        \end{equation}
        The term $\mu_k \partial_{zz} \vbf_{k,h}  /\varepsilon^2$ is thus of order $1$ and cannot be neglected in the equation.
    \end{remark}

\subsubsection{Averaging the equations}

    Now, we perform the averaging of the mass conservation equation and momentum equations for each fluid in the vertical direction, therefore obtaining a two-dimensional system. We approximate the boundary terms which appear during the integration by terms which only depend on the averaged unknowns in order to close the system. We keep trace of the approximation errors during the whole averaging process. We perform the averaging for the first fluid, the second fluid being treated in the same way.

    First, average the mass conservation equations~\cref{1_rescaled_cm}. We multiply the mass conservation equation~\cref{1_rescaled_cm} by $\chi_1$ and we integrate it with respect to $z \in (0,1)$ --- which is the same as integrating the equation with respect to $z \in (0,\alpha_1)$:
    \[ \int_0^1 \chi_1 \left(  \partial_t \rho_1 + \nabla_h \cdot (\rho_1\vbf_{1,h}) + \partial_z (\rho_1 w_1)\right) (z)\d z =0.\]
    Since $\chi_1 \partial_t \rho_1 = \partial_t(\chi_1 \rho_1)-\rho_1 \partial_t \chi_1$ with $\partial_t \chi_1 = \partial_t\alpha_1 \delta_{z=\alpha_1}$, 
    \[ \int_0^1 \chi_1 \partial_t \rho_1(z) \d z = \partial_t \int_0^{\alpha_1} \rho_1(z) \d z - \rho_1(\alpha_1) \partial_t \alpha_1= \partial_t (\alpha_1 \bar{\rho_1}) - \rho_1(\alpha_1) \partial_t \alpha_1.\]
    Proceeding similarly for the second term of the equation, we get that
    \begin{multline*}
        \partial_t (\alpha_1 \bar{\rho_1})-\rho_1 (\alpha_1)\partial_t \alpha_1 + \nabla_h \cdot (\alpha_1 \bar{\rho_1\vbf_{1,h}})- \rho_1\vbf_{1,h}(\alpha_1)\cdot \nabla_h (\alpha_1)\\
        + \rho_1w_1(\alpha_1) - \rho_1 w_1(0)=0.
    \end{multline*}
    Using the fraction equation~\cref{1_rescaled_fraction}, the bottom Navier condition~\cref{1_rescaled_bottom} and the relation $\bar{\rho_1\vbf_{1,h}}=\bar{\rho_1}\favre{\vbf_{1,h}}$, we finally get the averaged mass conservation equation~\cref{MR-cm} for the first fluid.

    \noindent The reasoning is the same for the second fluid, the equation being multiplied this time by $\chi_2$ and using the Navier condition~\cref{1_rescaled_top}. Note that the averaged mass conservation equations are exact: we did not make any approximations in order to get them.

    The averaging of the momentum equations is more complex and is detailed in the following proposition.

    \begin{proposition}
        The closed horizontal momentum equations read
        \begin{multline}
            \forall k\in \{1,2\}, \qquad \partial_t (\alpha_k \bar{\rho_k }\favre{\vbf_{k,h} }) + \nabla_h \cdot (\alpha_k \bar{\rho_k } \favre{\vbf_{k,h} } \otimes \favre{\vbf_{k,h} })
            + \nabla_h(\alpha_k \bar{p_k })\\
            =p_i \nabla_h \alpha_k-\frac{\kappa_k}{\varepsilon} \favre{\vbf_{k,h} }+{(-1)}^k\frac{\kappa_i}{\varepsilon}(\favre{\vbf_{1,h} }-\favre{\vbf_{2,h} })+ \rcal,
        \label{4_qdm_avg}
        \end{multline}
        where \[\rcal=O\Big(\mu+\kappa_i \varepsilon + \frac{\kappa_i(\kappa + \varepsilon)+\kappa(\kappa+\varepsilon)}{\mu} + \frac{{(\varepsilon \kappa + \varepsilon^2)}^2}{\mu^2} \Big).\]
    \end{proposition}

    \begin{proof}
        Once again, we develop the proof for $k=1$. Multiply the horizontal momentum equation~\cref{1_rescaled_qdm_h} by $\chi_1$ and integrate it with respect to $z \in (0,1)$. Using the fraction equation~\cref{1_rescaled_fraction}, we obtain
        \begin{multline*}
            \partial_t (\alpha_1 \bar{\rho_1}\favre{\vbf_{1,h}}) + \nabla_h \cdot (\alpha_1 \bar{\rho_1}\favre{\vbf_{1,h}\otimes \vbf_{1,h}})+\nabla_h (\alpha_1 \bar{p_1})\\
            =p_1(\alpha_1)\nabla_h(\alpha_1) + \frac{\mu_1}{\varepsilon^2} \partial_z \vbf_{1,h}(\alpha_1) -  \frac{\mu_1}{\varepsilon^2} \partial_z \vbf_{1,h}(0)\\
            + \int_0^{\alpha_1} \big( (\lambda_1+\mu_1) \nabla_h(\nabla \cdot \vbf_{1}) + \mu_1 \Delta_h \vbf_{1,h}\big)(z) \d z.
        \end{multline*}
        Thanks to the Navier condition~\cref{1_rescaled_bottom}, the equation reads 
        \begin{multline}
            \partial_t (\alpha_1 \bar{\rho_1 }\favre{\vbf_{1,h} })+\nabla_h\cdot (\alpha_1 \bar{\rho_1 }\favre{\vbf_{1,h} \otimes \vbf_{1,h} })+\nabla_h (\alpha_1 \bar{p_1 })=p_1 (\alpha_1) \nabla_h \alpha_1\\
            - \frac{\kappa_1}{\varepsilon}\vbf_{1,h} (0)+ \frac{\mu_1}{\varepsilon^2}\partial_z \vbf_{1,h} (\alpha_1) +\int_0^{\alpha_1}\big((\lambda_1+\mu_1)\nabla_h(\nabla \cdot \vbf_1 )+\mu_1 \Delta_h\vbf_{1,h} \big)(z) \d z.
        \label{3_qdm_h_brut}
        \end{multline}
        The asymptotic Navier condition~\cref{3_asymptotic_Navier_tan} then gives
        \[ \frac{\mu_1}{\varepsilon^2} \partial_z \vbf_{1,h}(\alpha_1)=\frac{\kappa_i}{\varepsilon}(\vbf_{2,h}(\alpha_1)-\vbf_{1,h}(\alpha_1))+O(\mu+\kappa_i \varepsilon).\]
        Injecting this new term into~\cref{3_qdm_h_brut} causes an approximation of order $\mu$. The integral term of the equation being of the same order, we obtain
        \begin{multline}
            \partial_t (\alpha_1 \bar{\rho_1 }\favre{\vbf_{1,h} })+\nabla_h\cdot (\alpha_1 \bar{\rho_1 }\favre{\vbf_{1,h} \otimes \vbf_{1,h} })+\nabla_h (\alpha_1 \bar{p_1 })=p_1 (\alpha_1) \nabla_h \alpha_1\\
            -\frac{\kappa_1}{\varepsilon}\vbf_{1,h}(0)+\frac{\kappa_i}{\varepsilon}(\vbf_{2,h}(\alpha_1)-\vbf_{1,h}(\alpha_1))+O(\mu+\kappa_i \varepsilon).
        \end{multline}

        \noindent To close this equation, the pointwise values of the velocity and the pressure must be expressed thanks to their averaged values. 

        \begin{lemma}\label{lemma 3}
            The pointwise values of the velocity $\vbf_{1,h}$ can be approximated as follows:
            \begin{equation}
                \forall z \in (0,\alpha_1), \qquad \vbf_{1,h}(z)=\favre{\vbf_{1,h}}+O\left(\frac{\varepsilon \kappa + \varepsilon^2}{\mu}\right).
            \label{4_approx_velocity_avg}
            \end{equation}
            The pointwise values of the pressure can be approximated as follows:
            \begin{equation}
                p_1(\alpha_1)=p_i+O(\mu),
            \label{4_approx_pressure_avg}
            \end{equation}
            where $p_i$ is defined in~\cref{thm_MR}.
        
            \noindent The same approximations remain valid for $\vbf_{2,h}$ and $p_2$. 
        \end{lemma}

        \begin{proof}
            Recall that in~\cref{rmk_1_odg_velocity}, we noticed that 
            \[ \forall z \in (0,\alpha_1), \qquad \mu_1 \partial_{zz} \vbf_{1,h}(z)=O(\varepsilon^2).\]
            Moreover, the bottom boundary Navier condition~\cref{1_rescaled_bottom} gives us that 
            \[\mu_1\partial_z \vbf_{1,h}(0)=O(\kappa \varepsilon).\]
            Thanks to a Taylor-Lagrange expansion, 
            \begin{equation}
                \forall z \in (0,\alpha_1), \qquad \mu_1 \partial_z \vbf_{1,h}(z)=O(\kappa \varepsilon+\varepsilon^2).
            \label{4_odg_dz_velocity}
            \end{equation}
            Another Taylor-Lagrange expansion gives
            \begin{align*}
                \forall (z,z')\in {(0,\alpha_1)}^2, \qquad \vbf_{1,h}(z)&=\vbf_{1,h}(z') + (z-z') \partial_z \vbf_{1,h}(z')+O(\varepsilon^2/\mu),\\
                &=\vbf_{1,h}(z') + O\left(\frac{\varepsilon \kappa + \varepsilon^2}{\mu}\right).
            \end{align*}
            Taking the Favre average of the previous equation with respect to $z'$ yields~\cref{4_approx_velocity_avg}.
            The reasoning is the same for $\vbf_{2,h}$ with $z \in (\alpha_1,1)$. We use the top Navier condition~\cref{1_rescaled_top} to get the estimate. 

            In the case of the pressure, the vertical component of the momentum equations~\cref{3_asymptotic_qdm_z} provide the estimates 
            \[ \left\{ \begin{array}{lc}
            \forall z \in (0,\alpha_1), & \partial_z p_1(z)=O(\mu),\\
            \forall z \in (\alpha_1,1), & \partial_z p_2(z)=O(\mu).
            \end{array}\right.\]
            Therefore, the pointwise values of the pressures can be approximated at order $\mu$ by their average. 
            However, interface condition~\cref{3_asymptotic_tensor_z} yields
            \[ p_1(\alpha_1)=p_2(\alpha_1)+O(\mu).\]
            Therefore, $\bar{p_1}=\bar{p_2}+O(\mu)$ and
            \[ p_1(\alpha_1)=p_i+O(\mu),\]
            where $p_i$ is defined in~\cref{thm_MR}. Hence~\cref{4_approx_velocity_avg} and~\cref{4_approx_pressure_avg}.
        \end{proof}

        \begin{remark}
            The computations above allow more freedom in the definition of $p_i$ than being a convex combination of the averaged pressures. Indeed, the proposition is still valid for any $p_i$ such that $\bar{p_1}-p_i=O(\mu)$.
        \end{remark}

        Approximating these different pointwise values by their (Favre) average yields
        \begin{multline*}
            \partial_t (\alpha_1 \bar{\rho_1 }\favre{\vbf_{1,h} })+\nabla_h\cdot (\alpha_1 \bar{\rho_1 }\favre{\vbf_{1,h} \otimes \vbf_{1,h} })+\nabla_h (\alpha_1 \bar{p_1 })\\
            =p_i \nabla_h \alpha_1-\frac{\kappa_1}{\varepsilon}\favre{\vbf_{1,h}} +\frac{\kappa_i}{\varepsilon}(\favre{\vbf_{2,h}}-\favre{\vbf_{1,h}})\\
            +O\left(\mu+\kappa_i \varepsilon+\frac{\kappa_i(\kappa+\varepsilon)+\kappa(\kappa+\varepsilon)}{\mu}\right).
        \end{multline*}

        \begin{remark}
            Replacing the pointwise values of the velocities by their Favre average in the Navier interface condition~\cref{3_asymptotic_Navier_tan} and using~\cref{4_odg_dz_velocity}, we have 
            \begin{equation}
                \favre{\vbf_{2,h}}-\favre{\vbf_{1,h}}=O\left( \frac{\kappa+\varepsilon+\mu \varepsilon}{\kappa_i}+\frac{\kappa \varepsilon + \varepsilon^2}{\mu}+\varepsilon^2\right).
            \label{4_diff_velocity}
            \end{equation}
        \end{remark}

        \noindent Finally, we have to express the Favre average $\favre{\vbf_{1,h}\otimes \vbf_{1,h}}$ with respect to the tensor product of the Favre averages. This way, the resulting system will only depend on the averages of the density, the velocity and the pressure. 

        \begin{lemma}
            The Favre average of the tensor product of the velocity $\vbf_{1,h}$ with itself is approximated by the tensor product of its Favre averages.
            \begin{equation}
                \favre{\vbf_{1,h}  \otimes \vbf_{1,h} }=\favre{\vbf_{1,h} } \otimes \favre{\vbf_{1,h} }+O\Big( \frac{{(\varepsilon \kappa + \varepsilon^2)}^2}{\mu^2}\Big).
            \label{4_approx_tensor_product}
            \end{equation}
        \end{lemma}

        \begin{proof}
            As shown in the previous lemma, 
            \[ \forall z \in (0,\alpha_1), \qquad \vbf_{1,h}(z)=\favre{\vbf_{1,h}}+f(z),\]
            \begin{flalign*}
                \text{where} && f=O\Big( \frac{\varepsilon \kappa + \varepsilon^2}{\mu}\Big) \quad \text{and} \quad \favre{f}=0.&&
            \end{flalign*}
            We can thus develop the tensor product $\vbf_{1,h} \otimes \vbf_{1,h} $ as
            \[ \vbf_{1,h}  \otimes \vbf_{1,h} =\favre{\vbf_{1,h} } \otimes \favre{\vbf_{1,h} } + 2\favre{\vbf_{1,h} } \otimes f + f \otimes f.\]
            Since $f \otimes f = O({(\kappa\varepsilon+\varepsilon^2)}^2/\mu^2)$, taking the Favre average of the equation above yields~\cref{4_approx_tensor_product}.
        \end{proof}

        Taking all of the above into account, we obtain the averaged momentum equation~\cref{4_qdm_avg}. For the second fluid, the reasoning is the same once we have rewritten the tangential Navier condition using~\cref{remark 1: tangential stress}.
    \end{proof}

\subsubsection{Closing the system}

    To get the set of equations that is presented in the main result, one must make sure that the momentum equations~\cref{4_qdm_avg} remain valid when $\varepsilon$ becomes very small. Three conditions arise:
    \begin{enumerate}
    \item $\rcal \limit{\varepsilon}{0} 0$, which means that the approximation errors vanish when $\varepsilon$ tends to $0$;
    \item $\favre{\vbf_{1,h}}-\favre{\vbf_{2,h}}$ must be well defined (i.e finite) when $\varepsilon$ tends to $0$;
    \item the friction terms in the momentum equations must have a limit when $\varepsilon$ tends to $0$ that is independent of $\varepsilon$.
    \end{enumerate}
    These conditions are the reason for the particular scaling requirements for the viscosity and friction coefficients presented in~\cref{MR_profile_viscosity_friction}. Indeed, let us prescribe viscosity and friction scalings as follows:
    \begin{equation}
        \forall k\in \{1,2\}, \qquad \mu_k=\widehat{\mu_k}\varepsilon^\tau, \quad \kappa_k=\widehat{\kappa_k}\varepsilon^\xi,\quad \kappa_{i}=\widehat{\kappa_{i}} \varepsilon^\zeta, 
    \end{equation}
    where $\tau,\xi,\zeta \in \R_+$. Rewriting the approximation error $\rcal$ with the previous scalings yields 
    \[ \rcal = O(\varepsilon^\tau+\varepsilon^{1+\zeta}+\varepsilon^{\xi+\zeta-\tau}+\varepsilon^{1+\zeta-\tau}+\varepsilon^{2\xi-\tau}+\varepsilon^{1+\xi-\tau}+\varepsilon^{2+2\xi-2\tau} + \varepsilon^{3+\xi-2\tau} + \varepsilon^{4-2\tau}).\]
    Satisfying the first condition adds the following constraints on the viscosity and friction scalings:
    \[ \zeta+\xi > \tau, \quad 1+\zeta >\tau, \quad 1+\xi > \tau, \quad 2\xi >\tau \text{ and } 2>\tau.\]

    Next, in order for the averaged model to be valid, the difference between the averaged velocities must remain finite when $\varepsilon$ becomes very small. Looking back at the relation~\cref{4_diff_velocity} and replacing the viscosities and friction coefficients by their scalings yields
    \[ \favre{\vbf_{1,h}}-\favre{\vbf_{2,h}}=O(\varepsilon^{\xi-\zeta}+\varepsilon^{1-\zeta}+\varepsilon^{1+\tau-\zeta}+\varepsilon^{1+\xi-\tau}+\varepsilon^{2-\tau}+\varepsilon^2).\]
    The second condition is thus satisfied if 
    \[ \xi \geqslant \zeta, \quad 1 \geqslant \zeta \text{ and } 1+\tau \geqslant \zeta.\]

    Finally, for the model to be valid, it must be independent of $\varepsilon$. Since $\favre{\vbf_{1,h}}-\favre{\vbf_{2,h}}$ is finite, we must ensure that the friction terms $\frac{\kappa_i}{\varepsilon}(\favre{\vbf_{2,h} }-\favre{\vbf_{1,h} })$ and $\frac{\kappa_k}{\varepsilon} \favre{\vbf_{k,h} }$ for $k\in \{1,2\}$ are well defined at the limit $\varepsilon \to 0$. This condition is satisfied if
    \[ \xi \geqslant 1 \text{ and } \zeta \geqslant 1.\]
    Taking into account all the different constraints results in the scalings~\cref{MR_profile_viscosity_friction}:
    \[\forall k\in \{1,2\}, \qquad \mu_k=\widehat{\mu_k}\varepsilon^\tau, \; 0<\tau < 2, \quad \kappa_k=\widehat{\kappa_k}\varepsilon^\xi, \; \xi\geqslant 1 \text{ and } \kappa_{i}=\widehat{\kappa_{i}} \varepsilon.\]

    \noindent Dropping the errors yields equations~\cref{MR-cm}-\cref{MR-qdm}. Equation~\cref{MR-sum-alpha} is given by the definition of the volume fractions. Equation~\cref{MR-barotropy-equality} has been proven in the demonstration of~\cref{lemma 3}.

    At this stage, this system is underdetermined. We indeed have eight unknowns and six equations. It remains to adapt the barotropy laws to the new set of equations to get equations~\cref{MR-barotropy}. Let us prove this result in the case $k=1$. 

    \noindent Using~\cref{3_asymptotic_qdm_z}, we know that $\partial_z (p_1(\rho_1))=O(\mu)$. Let $z^* \in (0,\alpha_1)$ such that $\rho_1(z^*)=\bar{\rho_1}$. Thanks to a Taylor-Lagrange expansion, 
    \[ \forall z \in (0,\alpha_1), \qquad p_1(\rho_1(z))=p_1(\rho_1(z^*))+O(\mu).\]
    Averaging the previous equation with respect to $z$ and dropping the error yields~\cref{MR-barotropy}.

    \begin{remark}
        Compared to the Navier-Stokes system of the first section, the limit of the fraction equation~\cref{1_fraction_advected_form} is not considered. Indeed, the volume fractions $\alpha_1$ and $\alpha_2$ are deduced from the equality of the averaged pressures~\cref{MR-barotropy-equality} and the saturation condition~\cref{MR-sum-alpha}.
    \end{remark}

\subsection{Recovering the one-velocity one-pressure model}

The choice of prescribing a Navier friction condition at the interface between the fluids is motivated by the desire to derive a two-velocity model. As mentioned in~\cref{rmk:about_interface_conditions}, other choices for the interface conditions are possible. If we allow the friction coefficient $\kappa_i$ to decrease slower than $\varepsilon$ (i.e taking $\kappa_i = \widehat{\kappa_i}\varepsilon^\zeta$ with $\zeta <1$), then the difference between the averaged velocities~\cref{4_diff_velocity} tends to zero when $\varepsilon$ becomes small. We then recover the model which would have resulted from the derivation with Dirichlet interface conditions with respect to the relative velocity:
\begin{equation}
    (\vbf_1-\vbf_2)|_{z=D\alpha_1}=0.
\label{eq:reminder_Dirichlet}
\end{equation}
This condition imposes the continuity of the velocity at the interface and leads to a one-velocity, one-pressure model, which reads: 

\begin{theorem}
    Assume the same scalings for the viscosity and friction coefficients as in~\cref{thm_MR} (conditions~\cref{MR_profile_viscosity_friction}). Let $(\alpha_1,\rho_1,\vbf_1,p_1,\alpha_2,\rho_2,\vbf_{2},p_2)$ be a solution to the Navier-Stokes system with Navier boundary and Dirichlet interface conditions~\cref{1_NS_cm}-\cref{1_tensor_interface},~\cref{1_fraction_advected_form},~\cref{eq:reminder_Dirichlet}.

    \noindent Then the asymptotic system satisfies $\favre{\vbf_{1,h}}=\favre{\vbf_{2,h}}$ up to an error of order $\varepsilon^{2-\tau}$. Denoting $\favre{\vbf_h}$ this velocity, the functions $(\alpha_1,\alpha_2,\bar{\rho_1},\bar{\rho_2},\favre{\vbf_h},\bar{p_1},\bar{p_2})$ satisfy the following closed system: 
    \begin{flalign}
        \partial_t(\alpha_k \bar{\rho_k}) + \nabla_h \cdot (\alpha_k \bar{\rho_k} \favre{\vbf_{h}})&=0, \label{Dirichlet-cm}\\
        \partial_t (\bar{\rho}\favre{\vbf_h}) + \nabla_h \cdot (\bar{\rho} \favre{\vbf_{h}} \otimes \favre{\vbf_{h}})+ \nabla_h(\alpha_1 \bar{p_1}+\alpha_2\bar{p_2})
        &=-\delta_\xi \widehat{(\kappa_1+\kappa_2)}\favre{\vbf_{h}},\label{Dirichlet-qdm}\\
        \bar{p_k}&=p_k(\bar{\rho_k}), \label{Dirichlet-barotropy}\\
        \bar{p_1}&=\bar{p_2}, \label{Dirichlet-barotropy-equality}\\
        \alpha_1+\alpha_2&=1 \label{Dirichlet-sum-alpha}
    \end{flalign}
    where $k \in \{1,2\}$ and $\bar{\rho}=\alpha_1\bar{\rho_1}+\alpha_2\bar{\rho_2}$. Same as before, the momentum equation is satified up to an error of order $\varepsilon^\tau+\varepsilon^{2-\tau}+(1-\delta_\xi)\varepsilon^{\xi -1}$ and the equality of the pressures is satisfied up to an error of order $\varepsilon^\tau$.
    This system is symmetrizable and therefore hyperbolic.
\end{theorem}

\begin{remark}
    The equality of the velocities up to an error of order $\varepsilon^{2-\tau}$ stems from the Dirichlet condition associated with the estimate~\cref{4_approx_velocity_avg}. Indeed, since the velocities are equal at the interface, and since they don't vary much vertically, we have that 
    \[ \left\{ \begin{array}{rcl}\vbf_{1,h}(\alpha_1)&=&\vbf_{2,h}(\alpha_1),\\ 
        \forall k\in \{1,2\},\qquad     \vbf_{k,h}(\alpha_1)&=&\favre{\vbf_{k,h}}+O\Big( \frac{\kappa \varepsilon+\varepsilon^2}{\mu}\Big).
    \end{array}\right.\]
    Once combined, and developing every coefficient with respect to its scaling, 
    \[ \favre{\vbf_{1,h}}=\favre{\vbf_{2,h}}+O(\varepsilon^{1+\xi-\tau}+\varepsilon^{2-\tau}).\]
    Since $\xi\geqslant 1$, the equation is satisfied up to an error of order $\varepsilon^{2-\tau}$.
\end{remark}

\section{The Navier-Stokes-Fourier case}\label{sec:Navier-Stokes-Fourier}
This final part extends the results of the previous section to the non-barotropic case. The microscopic model must therefore be completed with two energy equations, and the equations of state must now include the specific internal energies. We reproduce the averaging process on these two energy equations and pay special attention to the heat transfer between the fluids. This will enable us to derive a two-velocity, one-pressure and two-temperature averaged model. 
\subsection{The Navier-Stokes-Fourier system}

Let $k \in \{1,2\}$. Let $\theta_k: \R_+^* \times \Omega_k \to \R$ be the temperature of fluid $k$, $e_k: \R_+^* \times \Omega \to \R$ its specific internal energy $k$ and $E_k = (1/2)\norm{\vbf_k}^2 + e_k$ its specific total energy. Let $\beta_k$ be the thermal conductivity of the fluid $k$. It is a constant parameter of the flow that might depend on $\varepsilon$, similarly to the viscosities and friction coefficients. The flow is now described by two compressible Navier-Stokes-Fourier systems and two equations of state:

\begin{align}
    \partial_t \rho_k + \nabla \cdot (\rho_k \vbf_k)&=0, \label{5_cm}\\
    \partial_t (\rho_k \vbf_k) + \nabla \cdot (\rho_k \vbf_k \otimes \vbf_k) + \nabla p_k &= \nabla \cdot \S_k,\label{5_qdm}\\
    \partial_t (\rho_k E_k)+\nabla \cdot (\vbf_k (\rho_k E_k + p_k)) &= \nabla \cdot (\S_k \vbf_k)+ \nabla\cdot (\beta_k \nabla \theta_k), \label{5_energy}\\
    p_k&=p_k(\rho_k, \rho_k e_k), \label{5_pressure}\\
    \theta_k&=\theta_k(\rho_k, \rho_k e_k). \label{5_temperature}
\end{align}
The pressure $p_k$ and the temperature $\theta_k$ are defined thanks to a complete equation of state which assumes the existence of a specific entropy $s_k$ such that
\begin{equation}
    \theta_k \d s_k=  p_k \d \left( \frac{1}{\rho_k}\right)+\d e_k.
    \label{5_thermodynamic_principle}
\end{equation}

We keep Navier boundary conditions on the top and bottom of the domain (see~\cref{1_Navier_bottom} and~\cref{1_Navier_top}) and the interface conditions~\cref{1_tensor_interface} and~\cref{1_Navier_interface}. Additional conditions are required to describe the heat flux at the boundary and interface. We prescribe the following Robin-type conditions, see for instance~\cite{hahn_heat_2012}:
\begin{align}
    \beta_1\partial_z \theta_1|_{z=0}&=0,\label{5_adiabatic_bottom}\\
    \beta_2\partial_z \theta_2|_{z=D}&=0,\label{5_adiabatic_top}\\
    -\beta_1\nabla \theta_1\cdot \nbf_{\Gamma_D}|_{z=D\alpha_1}&= h_c(\theta_1-\theta_2)|_{z=D\alpha_1},\label{5_heat_1_interface}\\
    -\beta_2\nabla \theta_2\cdot \nbf_{\Gamma_D}|_{z=D\alpha_1}&= h_c(\theta_1-\theta_2)|_{z=D\alpha_1},\label{5_heat_2_interface}
\end{align}
where $h_c$ is the contact conductance for the interface. The first two conditions depict the adiabatic nature of the flow: the system is therefore thermically isolated. The last two conditions describe the heat exchange at the interface. More precisely, the heat flux at the interface is continuous and is expressed thanks to~\cref{5_heat_1_interface} and~\cref{5_heat_2_interface}.

\subsection{The averaged model}

The averaging procedure of the previous section is still valid in the Navier-Stokes-Fourier case. The difference lies in the presence of energy equations for both fluids as well as new equations of state for the pressure and the temperature. In the following, we are therefore going to focus on averaging these elements, which works in a similar way than what has been done previously. The detailed averaged model is given below. 

\begin{theorem}\label{5-thm_MR}
    Assume the following scalings for the different coefficients in our system:
    \begin{equation}
        \forall k\in \{1,2\}, \qquad \left\{ \begin{array}{rll}
        \mu_k=\widehat{\mu_k}\varepsilon^\tau, &\lambda_k=\widehat{\lambda_k}\varepsilon^\tau, & 0<\tau < 2 ; \\
        \kappa_{i}=\widehat{\kappa_i} \varepsilon, & \kappa_k=\widehat{\kappa_k}\varepsilon^\xi, & \xi\geqslant 1; \\
        \beta_k = \widehat{\beta_k} \varepsilon^{\gamma}, && 0\leqslant\gamma <2.
        \end{array}\right.
    \label{5-MR_profile_viscosity_friction}
    \end{equation}

    \begin{flalign*}
    \text{Define} && \delta_\xi = \left\{ \begin{array}{rl}
        1 & \text{if } \xi =1,\\
        0 & \text{if } \xi >1.
    \end{array}\right. \text{ and } \delta_\gamma = \left\{ \begin{array}{rl}
        1 & \text{if } \gamma =0,\\
        0 & \text{if } \gamma >0.
    \end{array}\right. &&
    \end{flalign*}

    \noindent Let $(\alpha_1,\rho_1,\vbf_1,p_1,E_1,\theta_1,\alpha_2,\rho_2,\vbf_{2},p_2,E_2,\theta_2)$ be a strong solution to the Navier-Stokes system with Navier boundary and interface conditions~\cref{1_fraction_advected_form},\cref{5_cm}-\cref{5_heat_2_interface}. Assume that 
    \begin{equation}
        \partial_z \nabla_h \theta_k=O(\varepsilon^2/\beta).
        \label{5_MR_hyp_grad_theta}
    \end{equation}

    \noindent The functions $(\alpha_1,\bar{\rho_1},\favre{\vbf_{1,h}},\bar{p_1},\favre{E_1},\bar{\theta_1},\alpha_2,\bar{\rho_2},\favre{\vbf_{2,h}},\bar{p_2},\favre{E_2},\bar{\theta_2})$ satisfy the following closed system:
    \begin{align}
        &\partial_t(\alpha_k \bar{\rho_k}) + \nabla_h \cdot (\alpha_k \bar{\rho_k} \favre{\vbf_{k,h}})=0,\label{5-MR-cm}\\
        \begin{split}
        &\partial_t (\alpha_k \bar{\rho_k} \favre{\vbf_{k,h}}) + \nabla_h \cdot (\alpha_k \bar{\rho_k} \favre{\vbf_{k,h}} \otimes \favre{\vbf_{k,h}})+ \nabla_h(\alpha_k \bar{p_k})\\
        &= p_i \nabla_h \alpha_k -\delta_\xi \widehat{\kappa_k}\favre{\vbf_{k,h}}+{(-1)}^k\widehat{\kappa_i}(\favre{\vbf_{1,h}}-\favre{\vbf_{2,h}}),
        \end{split}\label{5-MR-qdm}\\
        \begin{split}
        &\partial_t(\alpha_k \bar{\rho_k} \favre{E_{k,h}})+\nabla_h \cdot (\alpha_k (\bar{\rho_k} \favre{E_{k,h}}+\bar{p_k}) \favre{\vbf_{k,h}}) + p_i\partial_t \alpha_k\\
        &= -\delta_\xi \widehat{\kappa_k}\norm{\favre{\vbf_{k,h}}}^2
        +{(-1)}^{k}\widehat{\kappa_i} (\favre{\vbf_{1,h}}-\favre{\vbf_{2,h}})\cdot \favre{\vbf_{k,h}}\\
        &+\delta_\gamma (\widehat{\beta_k} \nabla_h \cdot (\alpha_k \nabla_h \bar{\theta_k}))+{(-1)}^{k} \widehat{h_c} (\bar{\theta_1}-\bar{\theta_2}),
        \end{split} \label{5-MR-energy}\\
        &\bar{p_k}=p_k(\bar{\rho_k},\bar{\rho_k}\favre{e_k}) \label{5-MR-EoS-pressure}\\
        &\bar{\theta_k}=\theta_k(\bar{\rho_k},\bar{\rho_k}\favre{e_k}),\label{5-MR-EoS-temperature}\\
        &\bar{p_1}=\bar{p_2},\label{5-MR-pressures-equality}\\
        &\alpha_1+\alpha_2=1, \label{5-MR-sum-alpha}
    \end{align}
    where $k \in \{1,2\}$ and $\widehat{h_c}=h_c/\varepsilon$. 

    As in the barotropic case, equations~\cref{5-MR-qdm} and~\cref{5-MR-pressures-equality} are satisfied up to an error of order $\varepsilon^\tau+\varepsilon^{2-\tau}+(1-\delta_\xi)\varepsilon^{\xi-1}$ and $\varepsilon^\tau$ respectively. The averaged energy equation~\cref{5-MR-energy} is satisfied up to an error of order $\varepsilon^\tau + \varepsilon^{2-\tau}+\varepsilon^{2-\gamma}+(1-\delta_\gamma)\varepsilon^\gamma$. The averaged equations of state~\cref{5-MR-EoS-pressure}-\cref{5-MR-EoS-temperature} are satisfied up to an error of order $\varepsilon^\tau$ and $\varepsilon^{2-\gamma}$ respectively.\end{theorem}

\begin{remark}
Replacing the equations of state~\cref{1_barotropy_laws} by the energy equations~\cref{5_energy} and the equations of state~\cref{5_pressure}-\cref{5_temperature} does not change the fact that the asymptotic averaged system has only one pressure.
\end{remark}

\begin{remark}
    The scaling of the thermal conductivity can be expressed in terms of the Péclet number $\text{Pe}=\text{Re}\text{Pr}$, where $\text{Pr}=c_p \mu/\beta$ is the Prandtl number, $c_p$ is the specific heat and $\text{Re}$ is the Reynolds number. In engineering applications, the Péclet number is often much larger than one, which is consistent with the fact that the Reynolds number is very large in our work. The prescribed scalings for the viscosity and the thermal conductivity enable us to browse different regimes for the Prandtl number.
\end{remark}

\subsection{Proof of~\texorpdfstring{\cref{5-thm_MR}}{5 thm MR}}
The proof of~\cref{5-thm_MR} is similar to the one in the barotropic case. In particular, the averaging of the mass conservation and momentum equations is the same. In this proof, we will only detail the specificities of the Navier-Stokes-Fourier case: rescaling the new unknowns and the energy equations, averaging of the equations of state and integrating the energy equations.

\subsubsection{Rescaling of the energy equations and the thermal flux conditions}

Following the procedure of the barotropic case, it is now necessary to rescale and simplify (by neglecting some terms, as in~\cref{subsec:asym-analysis}) the energy equation~\cref{5_energy}. We introduce the following characteristic dimensions: $\ecal=U^2$ for the specific energy, $\Theta$ for the temperature, $B=RU^3L/\Theta$ for the thermal conductivity and $H=RU^3/\Theta$ for the contact conductance. We then define the following rescaled variables: $\tilde{E}=E/\ecal$, $\tilde{\theta}=\theta/\Theta$, $\tilde{\beta}=\beta/B$ and $\tilde{h_c}=h_c/H$. Then, we remove of the $\tilde{\cdot}$ notation for the sake of clarity.

The rescaling procedure can be done term by term. Let us start with the left-hand side of the energy equation~\cref{5_energy}. Denoting $E_{k,h}=(1/2)\norm{\vbf_{k,h}}^2 + e_k$ for $k\in \{1,2\}$, the rescaled left-hand side of the energy equation reads
\[
\partial_t (\rho_k E_{k,h})+\nabla \cdot (\vbf_{k} (\rho_k E_{k,h} + p_k))
+ \varepsilon^2 (\partial_t (\rho_k w_k^2/2) + \nabla \cdot (\vbf_k (\rho_k w_k^2/2))).
\]
The terms of order $\varepsilon^2$ are neglected in the rest of the derivation.

Rescaling the thermal diffusion term is straightforward and yields 
\[\beta_k \Delta_h \theta_k + \frac{\beta_k}{\varepsilon^2} \partial_{zz} \theta_k.\]

The viscous term is the most technical one. 
\begin{lemma}\label{5-lemma-rescaling-viscous-term}
    After simplification, the rescaled viscous term in the energy equation~\cref{5_energy} reads
    \begin{equation}
        \frac{\mu_k}{\varepsilon^2} \partial_z (\vbf_{k,h}\cdot \partial_z \vbf_{k,h})+O(\mu).
    \label{5_lemma_viscous_term}
    \end{equation}
\end{lemma}

\begin{proof}
    Developping the viscous term of the energy equation~\cref{5_energy} yields
    \[
        \nabla \cdot (\S_k \vbf_k)=\nabla \cdot (\lambda_k (\nabla \cdot \vbf_k)\vbf_k) + \mu_k \nabla \cdot (2\D_k \vbf_k)
    \]
    The bulk term does not need to be detailed as the rescaling does not affect its expression. Since
    \begin{equation*}
       2\D_k \vbf_k=\begin{pmatrix}
            2u \partial_x u + v(\partial_y u + \partial_x v)+w(\partial_z u + \partial_x w)\\
            u(\partial_x v + \partial_y u) + 2v \partial_y v + w(\partial_z v + \partial_y w)\\
            u(\partial_x w + \partial_z u) + v(\partial_y w + \partial_z v) + 2w \partial_z w
        \end{pmatrix},
    \end{equation*}
    the rescaled version of its divergence reads
    \begin{multline*}
        \begin{pmatrix}
            \partial_x \\ \partial_y \\ \frac{1}{\varepsilon}\partial_z
        \end{pmatrix}
        \cdot
        \begin{pmatrix}
            2u \partial_x u + v(\partial_y u + \partial_x v)+w(\partial_z u + \varepsilon^2 \partial_x w)\\
            u(\partial_x v + \partial_y u) + 2v \partial_y v + w(\partial_z v + \varepsilon^2 \partial_y w)\\
            u(\varepsilon \partial_x w + \frac{1}{\varepsilon}\partial_z u) + v(\varepsilon\partial_y w + \frac{1}{\varepsilon} \partial_z v) + 2\varepsilon w \partial_z w
        \end{pmatrix}\\
        =\nabla_h \cdot ((\nabla_h \otimes \vbf_{k,h} + {(\nabla_h \otimes \vbf_{k,h})}^T)\vbf_{k,h})+\nabla_h \cdot (w_k \partial_z \vbf_{k,h} + \varepsilon^2 w_k \nabla_h w_k)\\
        +\partial_z (\vbf_{k,h}\cdot \nabla_h w_k)+\frac{1}{\varepsilon^2} \partial_z (\vbf_{k,h}\cdot \partial_z \vbf_{k,h})+2w_k\partial_z w_k.
    \end{multline*}
    
    \noindent Multiplying the previous term by $\mu_k$ and adding the bulk term yields the following expression for the rescaled viscous term:
    \begin{multline*}
        \lambda_k\nabla \cdot ( (\nabla \cdot \vbf_k)\vbf_k) + \mu_k \nabla_h \cdot ((\nabla_h \otimes \vbf_{k,h} + {(\nabla_h \otimes \vbf_{k,h})}^T)\vbf_{k,h})\\
        +\mu_k \nabla_h \cdot (w_k \partial_z \vbf_{k,h} + \varepsilon^2 w_k \nabla_h w_k)+\mu_k \partial_z (\vbf_{k,h}\cdot \nabla_h w_k)\\
        +\mu_k \frac{1}{\varepsilon^2} \partial_z (\vbf_{k,h}\cdot \partial_z \vbf_{k,h})+2\mu_k w_k\partial_z w_k.
    \end{multline*}
    The prevailing term in the previous expression is of order $\mu/\varepsilon^2$. The rest of the terms are of order $\mu$. Knowing that the use of Navier interface conditions cause an approximation error of order $\mu$, we neglect those terms right away to ensure simpler equations. Hence~\cref{5_lemma_viscous_term}.
\end{proof}

Finally, the rescaled version of the energy equation reads
\begin{multline}
    \partial_t (\rho_k E_{k,h})+\nabla \cdot (\vbf_{k} (\rho_k E_{k,h} + p_k))
    = \frac{\mu_k}{\varepsilon^2} \partial_z(\vbf_{k,h}\cdot \partial_z \vbf_{k,h}) \\
    + \beta_k \Delta_h \theta_k+\frac{\beta_k}{\varepsilon^2} \partial_{zz} \theta_k +O(\mu+\varepsilon^2).
    \label{5_rescaled_energy}
\end{multline}

The rescaled boundary and interface conditions~\cref{5_adiabatic_bottom}-\cref{5_heat_2_interface} are
\begin{align}
    \beta_1\partial_z \theta_1|_{z=0}&=0,\label{5_rescaled_adiabatic_bottom}\\
    \beta_2\partial_z \theta_2|_{z=1}&=0,\label{5_rescaled_adiabatic_top}\\
    \big(\varepsilon\beta_1\nabla_h \alpha_1\cdot \nabla_h \theta_1-\frac{\beta_1}{\varepsilon}\partial_z \theta_1\big)\big|_{z=\alpha_1}&= h_c(\theta_1-\theta_2)|_{z=\alpha_1},\label{5_rescaled_heat_1_interface}\\
    \big(\varepsilon\beta_2\nabla_h \alpha_1\cdot \nabla_h \theta_2-\frac{\beta_2}{\varepsilon}\partial_z \theta_2\big)\big|_{z=\alpha_1}&= h_c(\theta_1-\theta_2)|_{z=\alpha_1}.\label{5_rescaled_heat_2_interface}
\end{align}

\begin{remark}\label{5-remark_scaling_temperature}
    Multiplying equation~\cref{5_rescaled_energy} by $\varepsilon^2$, we get that for $k\in \{1,2\}$, 
    \[\beta_k \partial_{zz} \theta_k(z)=-\mu_k \partial_z(\vbf_{k,h}\cdot \partial_z \vbf_{k,h})+O(\varepsilon^2+\mu \varepsilon^2+\varepsilon^4). 
    \]
    Since $\mu_k \partial_z(\vbf_{k,h}\cdot \partial_z \vbf_{k,h}) = \mu_k \norm{\partial_z \vbf_{k,h}}^2 + \mu_k \vbf_{k,h}\cdot \partial_{zz} \vbf_{k,h}$, we can use the estimates~\cref{2_odg_dzz_velocity} and~\cref{4_odg_dz_velocity} and remember the prescribed scalings for the friction and the viscosity coefficients to get that
    \begin{equation}
    \beta_k \partial_{zz}\theta_k (z)=O(\varepsilon^2) \qquad \text{in } \Omega_k.
    \label{5_estimate_temperature_d2}
    \end{equation}
    Using the boundary conditions~\cref{5_rescaled_adiabatic_bottom} and~\cref{5_rescaled_adiabatic_top}, we get that for $k\in \{1,2\}$,
    \begin{equation}
        \beta_k \partial_{z}\theta_k (z)=O(\varepsilon^2) \qquad \text{in } \Omega_k.
        \label{5_estimate_temperature_d1}
        \end{equation}
    Therefore, for $k \in \{1,2\}$,
    \begin{equation}
        \theta_k (z)=\bar{\theta_k}+O(\varepsilon^2/\beta) \qquad \text{in } \Omega_k.
        \label{5_estimate_temperature}
        \end{equation}
    These estimates along with the interface conditions~\cref{5_rescaled_heat_1_interface} and~\cref{5_rescaled_heat_2_interface} tell us that 
    \begin{equation}
        h_c(\theta_1-\theta_2)|_{z=\alpha_1}=O(\varepsilon).
        \label{5_estimate_hc}
    \end{equation}
\end{remark}

\subsubsection{Integration of the energy equations}

Let $k=1$. Since the integral is linear, we are going to detail the averaging of the different terms of the energy equation~\cref{5_rescaled_energy}. 

First, we integrate the left-hand side of the equation between $z=0$ and $z=\alpha_1$. Similarly to section 2.3, we use the fraction equation~\cref{1_rescaled_fraction} and the impermeability condition~\cref{1_rescaled_bottom} to get 
\begin{multline*}
\int_0^{\alpha_1} \big(\partial_t (\rho_1 E_{1,h})+\nabla \cdot (\vbf_{1} (\rho_1 E_{1,h} + p_1))\big) \d z= \partial_t (\alpha_1 \bar{\rho_1} \favre{E_{1,h}})\\+\nabla_h \cdot (\alpha_1 (\bar{\rho_1}\favre{E_{1,h}\vbf_{1,h}}+\bar{p_1\vbf_{1,h}}))-p_1(\alpha_1)\nabla_h \alpha_1+p_1(\alpha_1)w_1(\alpha_1).
\end{multline*}

\noindent Using once again the fraction equation~\cref{1_rescaled_fraction},
\begin{multline*}
    \int_0^{\alpha_1} \big(\partial_t (\rho_1 E_{1,h})+\nabla \cdot (\vbf_{1} (\rho_1 E_{1,h} + p_1))\big) \d z = \partial_t (\alpha_1 \bar{\rho_1} \favre{E_{1,h}})\\+\nabla_h \cdot (\alpha_1 (\bar{\rho_1}\favre{E_{1,h}\vbf_{1,h}}+\bar{p_1\vbf_{1,h}}))
    +p_1(\alpha_1)\partial_t \alpha_1.
    \end{multline*}
Replacing the velocities by their Favre average thanks to~\cref{4_approx_velocity_avg}, we get 
\begin{multline*}
    \int_0^{\alpha_1} \big(\partial_t (\rho_1 E_{1,h})+\nabla \cdot (\vbf_{1} (\rho_1 E_{1,h} + p_1))\big) \d z = \partial_t (\alpha_1 \bar{\rho_1} \favre{E_{1,h}})\\+\nabla_h \cdot (\alpha_1 (\bar{\rho_1}\favre{E_{1,h}}+\bar{p_1})\favre{\vbf_{1,h}})
    +p_1(\alpha_1)\partial_t \alpha_1 + O\Big( \frac{\varepsilon\kappa + \varepsilon^2}{\mu}\Big).
\end{multline*}

Next, we average the viscous term. Using the Navier interface condition~\cref{3_asymptotic_Navier_tan} and the bottom Navier condition~\cref{1_rescaled_bottom}, we get
\begin{multline*}
    \int_0^{\alpha_1} \frac{\mu_1}{\varepsilon^2} \partial_z(\vbf_{1,h}\cdot \partial_z \vbf_{1,h}) \d z = \frac{\mu_1}{\varepsilon^2} (\vbf_{1,h}(\alpha_1)\cdot \partial_z \vbf_{1,h}(\alpha_1)-\vbf_{1,h}(0)\cdot\partial_z \vbf_{1,h}(0))\\
    = -\frac{\kappa_i}{\varepsilon} (\vbf_{1,h}(\alpha_1)-\vbf_{2,h}(\alpha_1))\cdot \vbf_{1,h}(\alpha_1)-\frac{\kappa_1}{\varepsilon}\norm{\vbf_{1,h}(0)}^2+O(\mu + \kappa_i \varepsilon).
\end{multline*}
Replacing the pointwise values of the velocities by their Favre average using~\cref{4_approx_velocity_avg}, we get
\begin{multline*}
    \int_0^{\alpha_1} \frac{\mu_1}{\varepsilon^2} \partial_z(\vbf_{1,h}\cdot \partial_z \vbf_{1,h}) \d z = -\frac{\kappa_i}{\varepsilon} (\favre{\vbf_{1,h}}-\favre{\vbf_{2,h}})\cdot \favre{\vbf_{1,h}}-\frac{\kappa_1}{\varepsilon}\norm{\favre{\vbf_{1,h}}}^2\\
    +O\Big(\mu + \kappa_i \varepsilon + \frac{(\kappa+\kappa_i)(\kappa + \varepsilon)}{\mu}\Big).
\end{multline*}

Last, integrate the thermal diffusion term. Using the adiabatic boundary condition~\cref{5_adiabatic_bottom} and the interface condition~\cref{5_rescaled_heat_1_interface}, we get
\begin{align*}
    \int_0^{\alpha_1} \big(\beta_1 \Delta_h \theta_1 + \frac{\beta_1}{\varepsilon^2} \partial_{zz}\theta_1 \big) \d z &= \beta_1 \nabla_h \cdot (\alpha_1 \bar{\nabla_h \theta_1}) \\
    & \quad -\beta_1\nabla_h \alpha_1\cdot \nabla_h \theta_1(\alpha_1)+\frac{\beta_1}{\varepsilon^2}\partial_z \theta_1(\alpha_1),\\
    &= \beta_1 \nabla_h \cdot (\alpha_1 \bar{\nabla_h \theta_1}) - \frac{h_c}{\varepsilon} (\theta_1(\alpha_1)-\theta_2(\alpha_1)).
\end{align*}

Approximating the pointwise values of the temperatures by their averages thanks to~\cref{5_estimate_temperature}, we get
\[ \int_0^{\alpha_1} \big(\beta_1 \Delta_h \theta_1 + \frac{\beta_1}{\varepsilon^2} \partial_{zz}\theta_1 \big) \d z = \beta_1 \nabla_h\cdot  (\alpha_1 \bar{\nabla_h \theta_1})- \frac{h_c}{\varepsilon}(\bar{\theta_1}-\bar{\theta_2}) + O(\varepsilon^2/\beta).\]
The term $\bar{\nabla_h \theta_1}$ is unsatisfactory since it introduces a new unknown. To solve this problem, two strategies are possible. The first one is to express this term in terms of the averaged temperature $\theta_1$ using the lemma below. 
\begin{lemma}
    Assume that for $k \in \{1,2\}$ \[ \partial_z \nabla_h \theta_k = O(\varepsilon^2/\beta) \quad \text{in } \Omega_k.\]
    Then, for $k \in \{1,2\}$, 
    \begin{equation}
        \nabla_h \cdot (\alpha_k \bar{\nabla_h \theta_k})=\nabla_h \cdot (\alpha_k \nabla_h \bar{\theta_k})+O(\varepsilon^2/\beta).
        \label{5-closure-temperature}
    \end{equation}
\end{lemma}
\begin{proof}
    The proof is done with $k=1$. Since 
    \begin{equation}
        \alpha_1\nabla_h \bar{\theta_1} = \alpha_1\bar{\nabla_h \theta_1}+\nabla_h\alpha_1 (\theta_1(\alpha_1)- \bar{\theta_1}),
        \label{5_gradient_theta}
    \end{equation} 
    taking the horizontal divergence of the previous equation yields
    \begin{multline*}
        \nabla_h \cdot (\alpha_1 \nabla_h \bar{\theta_1})=\nabla_h \cdot ( (\alpha_1 \bar{\nabla_h \theta_1}))+\Delta_h \alpha_1 (\theta_1(\alpha_1)-\bar{\theta_1})\\
        +\nabla_h \alpha_1 \cdot ((\nabla_h \theta_1)(\alpha_1) + \partial_z \theta_1(\alpha_1) \nabla_h \alpha_1-\nabla_h \bar{\theta_1}).
    \end{multline*}
    Using the estimate~\cref{5_estimate_temperature},
    \begin{equation*}
        \nabla_h \cdot (\alpha_1 \nabla_h \bar{\theta_1})=\nabla_h \cdot ( (\alpha_1 \bar{\nabla_h \theta_1}))+\nabla_h \alpha_1 \cdot ((\nabla_h \theta_1)(\alpha_1)-\nabla_h \bar{\theta_1})+O(\varepsilon^2/\beta).
    \end{equation*}
        
    \begin{flalign}
    \text{Reusing~\cref{5_gradient_theta} and estimate~\cref{5_estimate_temperature},} && \nabla_h \bar{\theta_1}=\bar{ \nabla_h \theta_1} + O(\varepsilon^2/\beta).&&
    \end{flalign} 
    Therefore,
    \begin{equation*}
        \nabla_h \cdot (\alpha_1 \nabla_h \bar{\theta_1})=\nabla_h \cdot (\alpha_1 \bar{\nabla_h \theta_1})+\nabla_h \alpha_1 \cdot ((\nabla_h \theta_1)(\alpha_1)-\bar{\nabla_h \theta_1})+O(\varepsilon^2/\beta).
    \end{equation*}
    Finally, thanks to the assumption~\cref{5_MR_hyp_grad_theta}, we can obtain equation~\cref{5-closure-temperature}.
\end{proof}

The second way to solve the introduction of the unknown $\bar{\nabla_h \theta_1}$ is to scale the thermal conductivities as $\beta_k=\widehat{\beta_k}\varepsilon^\gamma$ for $k \in \{1,2\}$ with $0<\gamma<2$. This allows to get rid of the heat diffusion term up to an error of order $\varepsilon^\gamma$. Note that in this case, no assumption about the vertical variation of $\nabla_h \theta_k$ is required. To synthesize the results, we assume that the thermal conductivities are scaled with $0 \leqslant \gamma <2$ as in~\cref{5-MR_profile_viscosity_friction} and that the hypothesis~\cref{5_MR_hyp_grad_theta} is satisfied.

Assembling all the terms, we get the averaged energy equation
\begin{multline}
    \partial_t(\alpha_1 \bar{\rho_1} \favre{E_{1,h}})+\nabla_h \cdot (\alpha_1 (\bar{\rho_1} \favre{E_{1,h}}+\bar{p_1}) \favre{\vbf_{1,h}}) +p_1(\alpha_1) \partial_t \alpha_1  =-\frac{\kappa_1}{\varepsilon}\norm{\favre{\vbf_{1,h}}}^2\\
    +\frac{\kappa_i}{\varepsilon} (\favre{\vbf_{1,h}}-\favre{\vbf_{2,h}})\cdot \favre{\vbf_{1,h}}
    + \beta_1 \nabla_h (\alpha_1 \nabla_h \bar{\theta_1})
    - \frac{h_c}{\varepsilon} (\favre{\theta_1}-\favre{\theta_2})+\rcal,
\end{multline}
\begin{flalign*}
    \text{where} && \rcal = O\Big(\mu +\kappa_i \varepsilon + \varepsilon^2 +\frac{(\kappa + \kappa_i + \varepsilon)(\kappa + \varepsilon)}{\mu}+\frac{\varepsilon^2}{\beta}\Big).&&
\end{flalign*}

Since we have that the averaged pressures are equal up to an error of order $\mu$, they are still indistiguishable in the averaged energy equation. Moreover, similarly to the momentum equation, the pointwise value $p_1(\alpha_1)$ can be replaced by any convex combination of the averaged pressures, that we will denote $p_i$. 

Using the scalings~\cref{5-MR_profile_viscosity_friction}, the averaged energy equation~\cref{5-MR-energy} holds up to an error of order $\varepsilon^\tau + \varepsilon^{2-\tau} + \varepsilon^{2-\gamma}+(1-\delta_\gamma)\varepsilon^\gamma$.

\begin{remark}
    The scaling for the contact conductance $h_c$ has been prescribed in order to get a two-temperature model. Note that the only constraint that we have is given by~\cref{5_estimate_hc}. Chosing a contact conductance that decreases slower than $\varepsilon$ would therefore imply that the difference of temperatures at the interface decreases to zero when $\varepsilon$ becomes very small, resulting in a one-temperature averaged model (since the temperatures do not vary sharply in the vertical direction).
\end{remark}

\subsubsection{Obtaining the averaged equations of state}

The equations of state are averaged in a similar way than in the barotropic case. However, now the pressures and the temperatures depend on two variables, so the proof needs a little adaptation. Let us prove the result for the pressure $p_1$ (the proof for $p_2$ and the temperatures are similar).

Recall that thanks to the asymptotic vertical momentum equation, we have that $\partial_z p_1(\rho_1,\rho_1e_1)=O(\mu)$. Let $z^* \in (0,\alpha_1)$ such that $\rho_1(z^*)=\bar{\rho_1}$. Thanks to a Taylor-Lagrange expansion, we have that 
\[ \forall z \in (0,\alpha_1),\qquad p_1(\rho_1(z),\rho_1e_1(z))=p_1(\rho_1(z^*),\rho_1e_1(z))+O(\mu).\]
Then, consider $z' \in (0,\alpha_1)$ such that $\rho_1e_1(z')=\bar{\rho_1e_1}=\bar{\rho_1}\favre{e_1}$. A second Taylor-Lagrange expansion yields
\[ \forall z \in (0,\alpha_1),\qquad p_1(\rho_1(z^*),\rho_1e_1(z))=p_1(\rho_1(z^*),\rho_1e_1(z'))+O(\mu).\]
\begin{flalign*}
\text{Therefore,} && \forall z \in (0,\alpha_1),\qquad p_1(\rho_1(z),\rho_1e_1(z))&=p_1(\bar{\rho_1},\bar{\rho_1}\favre{e_1})+O(\mu).&&
\end{flalign*}
Averaging the previous equation with respect to $z$ and dropping the error yields~\cref{5-MR-EoS-pressure}.

The procedure is the same for the temperature, but using the estimate~\cref{5_estimate_temperature_d1}, we get that the approximation error is of order $\varepsilon^2/\beta$.

\paragraph*{Acknowledgements}

The two last authors would like to thank Meriem Bahhi who nurtured the preliminary reflexions of this work.

\bibliographystyle{siamplain}
\bibliography{references}
\end{document}

%% file: ex_shared.tex
\usepackage[utf8]{inputenc} 
\usepackage[T1]{fontenc}    
\usepackage[english]{babel}

\usepackage{enumitem}
\usepackage{amsmath, amssymb}
\usepackage{mathtools}
\usepackage{esvect}
\usepackage{stmaryrd}
\usepackage{esint}

\usepackage{tikz}
\usetikzlibrary{3d}
\usepackage{pgfplots}
\pgfplotsset{compat = newest}
\usepackage{graphicx}

\newcommand{\ecal}{\mathcal{E}}
\newcommand{\mcal}{\mathcal{M}}
\newcommand{\rcal}{\mathcal{R}}

\newcommand{\D}{\ensuremath{\mathbb{D}}}
\newcommand{\I}{\ensuremath{\mathbb{I}}}
\newcommand{\R}{\ensuremath{\mathbb{R}}}
\renewcommand{\S}{\ensuremath{\mathbb{S}}}
\newcommand{\T}{\ensuremath{\mathbb{T}}}

\newcommand{\nbf}{\ensuremath{\mathbf{n}}}
\newcommand{\tbf}{\ensuremath{\mathbf{t}}}
\newcommand{\ubf}{\ensuremath{\mathbf{u}}}
\newcommand{\vbf}{\ensuremath{\mathbf{v}}}
\newcommand{\xbf}{\ensuremath{\mathbf{x}}}

\renewcommand{\d}{\mathrm{d}}
\newcommand{\trans}[1]{{(#1)}^{\mathrm{T}}}

\newcommand{\favre}[1]{\langle{#1}\rangle}
\renewcommand{\bar}[1]{\overline{#1}}

\newcommand{\ubrace}[2]{\underset{#1}{\underbrace{#2}}}

\newcommand{\norm}[1]{\left| #1 \right|}
\newcommand{\Norm}[1]{\left\| #1 \right\|}

\newcommand{\limit}[2]{\xrightarrow[#1\rightarrow #2]{}}

\usepackage{lipsum}
\usepackage{amsfonts}
\usepackage{graphicx}
\usepackage{epstopdf}
\usepackage{algorithmic}
\ifpdf
  \DeclareGraphicsExtensions{.eps,.pdf,.png,.jpg}
\else
  \DeclareGraphicsExtensions{.eps}
\fi

\newsiamremark{remark}{Remark}
\newsiamremark{hypothesis}{Hypothesis}
\crefname{hypothesis}{Hypothesis}{Hypotheses}
\newsiamthm{claim}{Claim}
\newsiamremark{fact}{Fact}
\crefname{fact}{Fact}{Facts}

\headers{Averaged models for compressible stratified two-phase flows}{P. Le Vourc'h, K. Saleh, and N. Seguin}

\title{Averaged models for compressible two-phase stratified flows on thin domains\thanks{Submitted to the editors on \today.
\funding{N. Seguin’s work was co-funded by
the Agence Nationale de la Recherche (ANR) under the ANR MSMPhi project (ANR-23-CE40-0013) and the ANR HEAD project (ANR-24-CE40-3260).}}}

\author{Pierrick Le Vourc'h\thanks{IMAG, ANGUS, Univ Montpellier, Inria, CNRS, Montpellier, France
  (\email{pierrick.le-vourc-h@umontpellier.fr}, \email{nicolas.seguin@inria.fr}).}
\and Khaled Saleh\thanks{Institut de Mathématiques de Marseille, Université d'Aix-Marseille, CNRS, Marseille, France
  (\email{khaled.saleh@univ-amu.fr}).}
\and Nicolas Seguin\footnotemark[2]}

\usepackage{amsopn}
